\documentclass[a4paper]{amsart}

\usepackage[breaklinks,colorlinks]{hyperref}
\usepackage{begnac}

\usepackage{stmaryrd}
\usepackage{eucal}

\usepackage{mathtools}
\newcommand{\sub}{\subseteq}
\DeclarePairedDelimiter\set{\{}{\}}
\DeclareMathOperator{\Sym}{Sym}
\DeclareMathOperator{\UCB}{UCB}
\newcommand{\actson}{\curvearrowright}
\renewcommand{\cl}[2][]{\overline{#2}^{#1}}

\newcommand{\dslash}{\mathord\sslash} 

\usepackage{xcolor}
\definecolor{dark-red}{rgb}{0.7,0.15,0.15}
\definecolor{dark-green}{rgb}{0.15,0.6,0.15}
\definecolor{medium-blue}{rgb}{0,0,0.5}
\hypersetup{
  colorlinks,
  linkcolor={dark-red}, citecolor={dark-green}, urlcolor={medium-blue}}

\usepackage{enumitem}
\setlist[enumerate,1]{label=(\roman*), font=\normalfont}

\numberwithin{equation}{section}

\begin{document}

\title{Locally $\aleph_0$-categorical theories and locally Roelcke precompact groups}

\author{Itaï \textsc{Ben Yaacov}}
\address{Itaï \textsc{Ben Yaacov} \\
  Université Claude Bernard Lyon 1 \\
  Institut Camille Jordan, CNRS UMR 5208 \\
  43 boulevard du 11 novembre 1918 \\
  69622 Villeurbanne Cedex \\
  France}
\urladdr{\url{https://math.univ-lyon1.fr/~begnac/}}

\author{Todor \textsc{Tsankov}}
\address{Todor \textsc{Tsankov} \\
  Université Claude Bernard Lyon 1 \\
  Institut Camille Jordan, CNRS UMR 5208 \\
  43 boulevard du 11 novembre 1918 \\
  69622 Villeurbanne Cedex \\
  France}
\urladdr{\url{https://math.univ-lyon1.fr/~tsankov/}}

\thanks{The authors were partially supported by the ANR project MAS (ANR-25-CE40-5294).}

\date{\today}
\keywords{locally Roelcke precompact groups, locally $\aleph_0$-categorical structures, model theory, Polish groups}
\subjclass[2020]{03C15,03C66,22F50}

\begin{abstract}
  It is well-known that Polish Roelcke precompact groups are the groups that can be represented as automorphism groups of $\aleph_0$-categorical structures in continuous logic and that there is a precise correspondence between properties of the group and properties of the structure.
  The goal of this paper is to extend this correspondence to the classes of locally Roelcke precompact groups and locally $\aleph_0$-categorical structures, the latter of which we define here.

  We characterise locally Roelcke precompact groups in terms of their isometric actions.
  We define locally $\aleph_0$-categorical theories and structures, prove an appropriate version of the Ryll-Nardzewski theorem, and identify the Polish locally Roelcke precompact groups as the automorphism groups of such structures.
  In all locally $\aleph_0$-categorical structures, there is a definable metric, which we call localising, and which captures the coarse geometric structure of the corresponding automorphism group.
  We show that two locally $\aleph_0$-categorical structures are bi-interpretable if and only if their automorphism groups are isomorphic.
  Finally, we show that (the unit ball of) a Banach space is $\aleph_0$-categorical if and only if the corresponding affine space is locally $\aleph_0$-categorical (as a metric space).
\end{abstract}

\maketitle

\tableofcontents


\section*{Introduction}

Recall that a permutation group $G \leq \Sym(M)$ on a set $M$ is called \emph{oligomorphic} if the diagonal action $G \actson M^n$ has finitely many orbits for every $n$.
By a classical theorem of Ryll-Nardzewski, oligomorphic permutation groups are precisely the automorphism groups of $\aleph_0$-categorical structures and if $M$ is such a structure, the definable subsets in $M^n$ are the $G$-invariant sets.
This allows to study oligomorphic groups from different viewpoints and has led to a rich interaction between model theory, group theory, combinatorics, and computer science.

The  correspondence between groups and $\aleph_0$-categorical structures was generalised in \cite{BenYaacov-Tsankov:WAP} to the metric setting.
We  recall that a topological group $G$ is \emph{Roelcke precompact} if it is precompact in the Roelcke uniformity; equivalently, if for every non-empty, open $U \sub G$, there exists a finite $F \sub G$ such that $UFU = G$.
All compact groups are Roelcke precompact and other typical examples are given by the symmetric group $\Sym(A)$ for any set $A$, the unitary group of any Hilbert space, or the isometry group of the Urysohn sphere.

When $(A, d)$ is a metric space and $G \actson A$ is a group action by isometries, we let
\begin{equation*}
  A \dslash G = \set{\cl{Ga} : a \in A}
\end{equation*}
be the collection of orbit closures.
The set distance defines a metric on $A \dslash G$, which is complete if $A$ is complete.

The following is a combination of the Ryll-Nardzewski theorem for continuous logic \cite{BenYaacov-Berenstein-Henson-Usvyatsov:NewtonMS} and results from \cite{BenYaacov-Tsankov:WAP} and \cite{BenYaacov-Kaichouh:Reconstruction}.
\begin{fct*}
  Let $M$ be a separable, continuous logic structure and let $G = \Aut(M)$.
  Then the following are equivalent:
  \begin{enumerate}
  \item The structure $M$ is $\aleph_0$-categorical.
  \item The space $M \dslash G$ is compact, and any of the following equivalent conditions holds:
    \begin{enumerate}
    \item The space $M^n \dslash G$ is compact for every $n$.
    \item The group $G$ is Roelcke precompact.
    \end{enumerate}
  \end{enumerate}

  If these equivalent conditions hold, then, in addition:
  \begin{itemize}
  \item A predicate $P \colon M^n \to \bR$ is definable if and only if it is continuous and $G$-invariant.
  \item An $\aleph_0$-categorical structure $N$ is bi-interpretable with $M$ if and only if $\Aut(N) \cong G$ as topological groups.
  \end{itemize}
\end{fct*}

Since every Polish group $G$ can be represented as the automorphism group of a separable structure $M$ with $M \dslash G$ a singleton, this gives a characterisation of the Roelcke precompact groups: they are exactly the automorphism groups of $\aleph_0$-categorical structures.
Moreover, one can establish a precise correspondence between dynamical properties of $G$ on one side and model-theoretic properties of $M$ on the other.
See \cite{BenYaacov-Tsankov:WAP,Ibarlucia:DynamicalHierarchy,BenYaacov-Ibarlucia-Tsankov:Eberlein} for more details.

\medskip

The goal of the present paper is to establish a similar correspondence between natural generalisations of Roelcke precompact groups and $\aleph_0$-categorical structures, both of which may be qualified, for different reasons, as ``local''.

Locally Roelcke precompact groups, introduced by Rosendal~\cite{Rosendal:CoarseGeometry}, are topological groups in which the identity admits a Roelcke precompact neighbourhood.
They were also considered by Zielinski~\cite{Zielinski:LocallyRoelckePrecompact}, who characterised the locally Roelcke precompact groups as the ones whose Roelcke completion is locally compact and by Hallbäck~\cite{Hallback:PhD}, who studied some model-theoretic aspects.
Examples include the Roelcke precompact groups, the locally compact groups, isometry groups of metrically homogeneous graphs, affine isometry groups of Hilbert spaces, the isometry group of the infinite-dimensional hyperbolic space, and the isometry group of the Urysohn metric space.

On the model-theoretic side, we should like to call the structures from the previous paragraph \emph{locally} $\aleph_0$-categorical.
Their common feature is that models of the theory can be represented as disjoint unions at infinite distance of \emph{local components}, with no interaction between them, and the structure itself is the unique separable \emph{local model} (i.e., consisting of a single component).
The term ``local'' is borrowed from Hrushovski~\cite{Hrushovski:BeyondLascar}, where related notions are considered.

\medskip

The main new feature that distinguishes locally Roelcke precompact groups from the Roelcke precompact ones is the presence of coarse geometric considerations.
This means that one is no longer only concerned with what happens near the identity, but also at large distances.
Our main result concerning locally Roelcke precompact groups is a characterisation in terms of their isometric actions (see \autoref{thm:LocallyRoelckePrecompactProper}).
We recall that a metric space is \emph{proper} if all closed balls are compact and an action of a topological group $G$ on a metric space $(X, d)$ is \emph{coarsely proper} if for every $x \in X$ and $r \in \bR$, the set $\set{g \in G : d(x, gx) < r}$ is coarsely bounded in $G$ (see \cite{Rosendal:CoarseGeometry} for more details).
\begin{thm*}
  Let $X$ be a complete separable metric space and let $G \leq \Iso(X)$ be such that $X \dslash G$ is a proper metric space.
  Then the following are equivalent:
  \begin{enumerate}
  \item The group $G$ is locally Roelcke precompact and the action $G \curvearrowright X$ is coarsely proper.
  \item The quotient $X^n \dslash G$ is a proper metric space for every $n$.
  \end{enumerate}
\end{thm*}

A (non-trivial) coarse geometric structure is most conveniently coded by an (unbounded) metric.
Model theoretically, we can view an unbounded metric $d$ as a definable predicate, provided that we allow it to take values in the compact interval $[0,\infty]$.
This makes $d$ a \emph{generalised metric}, possibly attaining the value infinity in other models of the theory.

If $d$ is a definable generalised metric in a theory $T$, then the relation given by $d(a, b) < \infty$ on models of $T$ is \emph{open}, in the sense that it is given by an open set in the space of types $\tS_2(T)$, and it is an equivalence relation.
The converse of this is a metrisation theorem that will be very useful for us but which we also believe to be of independent interest (see \autoref{thm:DefinableMetrisation}).
\begin{thm*}
  Let $T$ be a theory in a separable language, and let $\sim$ be an equivalence relation on models of $T$ given by an open set of $2$-types.
  Then $T$ admits a definable generalised metric $d$, unique up to uniform and coarse equivalence, such that $d(a,b) < \infty$ if and only if $a \sim b$.
\end{thm*}

We define a \emph{weakly local theory} as one whose models can be presented as disjoint unions of \emph{local components} with no interaction between them.
Formally, this is a pair $(T, \sim)$, where $T$ is a theory and $\sim$ is an invariant equivalence relation on models of $T$, such that every $\sim$-class (\emph{local component}) of a model of $T$ is an elementary substructure and a map between models of $T$ is elementary if and only if it preserves $\sim$ and is elementary on each component separately (see \autoref{dfn:WeaklyLocalTheory}).
If the equivalence relation $\sim$ is of the form $d(x,y) < \infty$ for a definable generalised metric $d$, then $d$ is unique up to coarse and uniform equivalence, and we say that it is \emph{localising}.

A theory in a separable language is \emph{locally $\aleph_0$-categorical} if it is weakly local with respect to some equivalence relation $\sim$, and it admits a unique separable \emph{local model} (i.e., consisting of a unique component) up to isomorphism (see \autoref{dfn:LocallyA0Categorical}).
When $\sim$ is the trivial equivalence relation (by which all members are equivalent), this coincides with ordinary $\aleph_0$-categoricity.

In this paper, we do \emph{not} attempt to give a general definition of a local theory (or local logic) outside of the locally $\aleph_0$-categorical context.

In order to give a Ryll-Nardzewski-style characterisation of locally $\aleph_0$-categorical theories, we require a slight weakening of the notion of an isolated type.
Let $a$ be an $n$-tuple, and $p = \tp(a)$.
For $i,j < n$, say that $i \sim_p j$ if $\tp(a_i,a_j)$ is isolated.
We say that $p$ is \emph{locally isolated} if, first, $\sim_p$ is an equivalence relation on $n$, second, the restriction of $p$ to each equivalence class of indices is isolated, and third, $p$ is determined by $\sim_p$ and the restrictions to its equivalence classes (see \autoref{dfn:LocallyIsolated}).

With this definition, the Ryll-Nardzewski characterisation of $\aleph_0$-categorical theories, as these in which all types are isolated, generalises as follows (see \autoref{thm:LocalRyllNardzewski}).
\begin{thm*}
  A weakly local theory $(T, \sim)$ is locally $\aleph_0$-categorical if and only if it is complete and all types are locally isolated.

  Moreover, if these equivalent conditions hold, then: $a \sim b$ if and only if $\tp(a,b)$ is isolated; the unique separable local model of $T$ is its prime model; the isolated types in $\tS_n(T)$ form an open set for each $n$; and $T$ admits a localising metric, i.e., a definable generalised metric $d$ such that $d(a, b) < \infty$ if and only if $a \sim b$.
\end{thm*}

Say that a structure $M$ is \emph{locally $\aleph_0$-categorical} if it is the unique separable local model of its theory.
As we recalled in the beginning, $M$ is $\aleph_0$-categorical if and only if $M^n \dslash G$ is compact for all $n$, if and only if $G$ is Roelcke precompact and $M \dslash G$ is compact, where $G = \Aut(M)$.
An analogous characterisation of local $\aleph_0$-categoricity, solely in terms of the action of $G$, is not possible.
For example, $(\bZ, s)$, where $s$ is the successor function, and $(\bZ, <)$, are two structures with the same underlying set and (locally compact) automorphism group $(\bZ,+)$.
The first is locally $\aleph_0$-categorical (see \autoref{exm:ZSucc}). However, the second is not because the order is a non-local relation, in the sense that it allows interactions between distinct local components (\autoref{exm:ZOrd}).
The problem lies in the fact that the action of the automorphism group only provides information about the realised (equivalently, isolated) types, while the theory of $M$ admits non-isolated types (unless $M$ is actually $\aleph_0$-categorical).

To overcome this obstacle, we need to prescribe the behaviour of definable predicates at infinity.
To that end, we introduce the notion of proper $G$-invariance.
A unary predicate is \emph{properly $G$-invariant} if it is $G$-invariant, and a binary predicate $P$ is \emph{properly $G$-invariant} if it is $G$-invariant and $\lim_{g \to \infty} P(a, gb)$ exists for all $a, b \in M$ (here $g$ is an element of $G$ and $g \to \infty$ means that $g$ eventually leaves every coarsely bounded set).
The definition for higher arities is more complex and we postpone it to \autoref{dfn:ProperlyInvariant}.

A key technical tool is the construction of a metric structure $A^G$ from a complete, separable metric space $A$ together with a closed subgroup $G \leq \Iso(A)$ (see \autoref{dfn:GroupActionStructure}).
The structure $A^G$ is characterised by the following: it is the prime model of its theory, $\Aut(A^G) = G$, and the collection of definable predicates on $A^G$ is minimal given these two properties (\autoref{prp:GroupActionStructureMinimal}).

With this in hand, we can state the full analogue of the characterisation of $\aleph_0$-categorical structures. The following is a summary of our main results (see \autoref{cor:LocallyA0CategoricalAutGroup}, \autoref{thm:GroupActionStructure}, \autoref{thm:LocallyA0CategoricalDefinable}, \autoref{cor:LocallyA0CategoricalDefinitionEquivalent}, \autoref{thm:LocalCoquand}, \autoref{thm:LocallyA0CategoricalReduct}, and \autoref{cor:LocallyA0CategoricalSufficient}).
\begin{thm*}
  Let $M$ be a separable, continuous logic structure with a (possibly unbounded) metric $d$ and let $G = \Aut(M)$.
  Then the following are equivalent:
  \begin{enumerate}
  \item $M$ is locally $\aleph_0$-categorical and $d$ is a localising metric.
  \item \label{i:intro:bigth:cond} The metric space $M \dslash G$ is compact, every atomic formula in $M$ is properly $G$-invariant, and any of the following equivalent conditions holds:
    \begin{enumerate}[ref=(\roman{enumi}.\alph*)]
    \item \label{i:intro:bigth:cond:proper} The metric space $M^n \dslash G$ is proper for every $n$.
    \item The group $G$ is locally Roelcke precompact and the action $G \actson M$ is coarsely proper.
    \end{enumerate}
  \end{enumerate}

  If these equivalent conditions hold, we have in addition:
  \begin{itemize}
  \item The structure $M$ is bi-definable with $M^G$.
  \item A predicate $P \colon M^n \to [-\infty, \infty]$ is definable if and only if it is uniformly continuous and properly $G$-invariant.
  \item A locally $\aleph_0$-categorical structure $N$ is bi-interpretable with $M$ if and only if $\Aut(N) \cong G$ as topological groups.
  \end{itemize}
\end{thm*}

While most of these results are parallel to (and indeed, generalise) ones about $\aleph_0$-categorical structures, the local setting is rather more delicate.
An important obstacle is the fact that, unlike $\aleph_0$-categoricity, local $\aleph_0$-categoricity is not preserved under reducts (see \autoref{exm:ZBiColour}).

Applied to the action of a Polish group on its left completion, the theorem yields the following characterisation.
\begin{cor*}
  A Polish group is locally Roelcke precompact if and only if it is the automorphism group of some locally $\aleph_0$-categorical structure.
\end{cor*}

In practice, to check that a particular structure is locally $\aleph_0$-categorical, one usually verifies condition \autoref{i:intro:bigth:cond:proper} of the theorem.
In particular, for pure metric spaces, we have the following (see \autoref{cor:pure-metric-spaces-loca0}).
\begin{cor*}
  Let $(A, d)$ be a complete, separable, unbounded metric space, and let $G = \Iso(A)$.
  Then $(A, d)$ is locally $\aleph_0$-categorical if and only if $A \dslash G$ is compact and $A^n \dslash G$ is proper for every $n \in \bN$.
\end{cor*}

An important class of locally $\aleph_0$-categorical metric spaces is given by affine Banach spaces.
By the Mazur--Ulam theorem, every isometry of a Banach space is automatically affine, so it is natural to consider affine Banach spaces in the pure metric language.
For a Banach space $E$, let $E_1$ denote the closed unit ball equipped with the trace of the linear structure (this is the usual way to consider Banach spaces in continuous logic).
Many (unit balls of) Banach spaces are known to be $\aleph_0$-categorical (e.g., Hilbert spaces, $L^p$ spaces for $1 \leq p < \infty$, the Gurarij space).
The following (see \autoref{thm:BanachSpaceExample}) establishes, in particular, that the corresponding affine spaces are locally $\aleph_0$-categorical.

\begin{thm*}
  Let $E$ be a separable Banach space. Then $E$ (in the pure metric language) is locally $\aleph_0$-categorical if and only if the unit ball $E_1$ (with its linear structure) is $\aleph_0$-categorical.
\end{thm*}

A structure is \emph{absolutely categorical} in continuous logic (i.e., its theory has a unique model up to isomorphism) if and only if it is compact if and only if it is $\aleph_0$-categorical and its automorphism group is compact.
The analogous result in the local setting is the following (see \autoref{thm:LocallyCompact}).
\begin{thm*}
  Let $T$ be a locally $\aleph_0$-categorical theory, let $M$ be its prime model equipped with a localising metric $d$, and let $G = \Aut(M)$.
  Then the following are equivalent:
  \begin{enumerate}
  \item
    The metric space $(M, d)$ is proper.
  \item
    The group $G$ is locally compact.
  \item The structure $M$ is the unique local model of $T$ up to isomorphism.
  \end{enumerate}
  When $T$ is not $\aleph_0$-categorical, these are further equivalent to
  \begin{enumerate}[resume]
  \item The theory $T$ is $\kappa$-categorical for some (every) uncountable cardinal $\kappa$.
  \end{enumerate}
\end{thm*}


\section{Locally Roelcke precompact Polish groups}
\label{sec:locally-roelcke-prec}

Let $G$ be a Polish group.
By the Birkhoff--Kakutani theorem, it admits a compatible left-invariant metric, that we usually denote by $d_L$.
Any two such metrics are uniformly equivalent, giving rise to a unique (as a uniform space) left completion $\widehat{G}_L$.
For the following fact, see \cite{Roelcke-Dierlof:UniformStructures} and \cite{BenYaacov-Tsankov:WAP}.

\begin{fct}
  \label{fct:GroupActionCompletion}
  Let $G$ be a Polish group, $d_L$ a compatible, left-invariant metric on $G$, $(X, d)$ a metric space, and $G \curvearrowright X$ a continuous, isometric action.
Then:
  \begin{itemize}
  \item The action map $G \times X \rightarrow X$ extends uniquely to a continuous map $\widehat{G}_L \times \widehat{X} \rightarrow \widehat{X}$, denoted $(a,x) \mapsto ax$, where $\widehat{X}$ is the completion of $X$.
    For $a \in \widehat{G}_L$, the map $x \mapsto ax$ is isometric.
    For $x \in \widehat{X}$, the map $a \mapsto ax$ is uniformly continuous (with a continuity modulus dependent on $x$).
  \item In the special case where $X = (G,d_L)$, this defines a semigroup law $\widehat{G}_L \times \widehat{G}_L \rightarrow \widehat{G}_L$.
  \item In the general case, $\widehat{G}_L \times \widehat{X} \rightarrow \widehat{X}$ defines a semigroup action, by which we mean that $a(bx) = (ab)x$ for $a,b \in \widehat{G}_L$ and $x \in \widehat{X}$.
  \end{itemize}
\end{fct}

\begin{dfn}
  \label{dfn:OrbitClosureQuotient}
  Let $G$ be a group, let $(X, d)$ be a metric space, and let $G \curvearrowright X$ by isometries.
  For $x \in X$, we shall denote by $[x]$ the orbit closure $\overline{G x} \subseteq X$.
  The collection of all orbit closures will be denoted by $X \dslash G$.

  For $x,y \in X$, let $d\bigl(x, [y]\bigr)$ and $d\bigl( [x], [y]\bigr)$ denote the distances from a point to a set and between sets in $X$, respectively.
\end{dfn}

Let $G \curvearrowright X$ be an action as in \autoref{dfn:OrbitClosureQuotient}.
Since it is isometric, the relation $x \in \cl{Gy}$ is an equivalence relation and we have
\begin{gather*}
  d\bigl(x, [y] \bigr) = d\bigl( x, Gy \bigr) = d\bigl( [x], [y] \bigr)
\end{gather*}
for all $x,y \in X$.
It follows that the distance between sets defines a metric on $X \dslash G$, and if $X$ is complete, then so is $X \dslash G$.

\begin{lem}
  \label{lem:CoCompact}
  Let $G$ be a group acting by isometries on a complete metric space $X$.
  Then $X \dslash G$ is compact if and only if there exists a compact subset $K \subseteq X$ that meets every orbit closure in $X \dslash G$.
\end{lem}
\begin{proof}
  Let $\pi \colon X \to X \dslash G$ denote the quotient map.
  If $K \subseteq X$ is a compact set meeting every orbit closure, then $\pi(K) = X \dslash G$ is compact.

  Conversely, assume that $X \dslash G$ is compact.
  For each $k \in \bN$, let $P_k$ be a finite $2^{-k}$-net in $X \dslash G$.
  We may choose by induction on $k$ subsets $A_k \subseteq X$ such that $|A_k| = |P_k|$, $\pi(A_k) = P_k$, and $d(x,A_k) < 2^{-k}$ for each $x \in A_{k+1}$.
  Indeed, if $[z] \in P_{k+1}$, then there exists $y \in A_k$ such that $d\bigl( [z], [y]\bigr) = d\bigl( [z], y \bigr) < 2^{-k}$, and any $x \in [z]$ such that $d(x,y) < 2^{-k}$ will do as a preimage for $[z]$.

  The union $\bigcup_k A_k$ is totally bounded, so $K = \overline{\bigcup_k A_k}$ is compact.
  If $p \in X \dslash G$, then there exist $p_k \in P_k$ such that $p_k \to p$.
  Let $x_k \in A_k$ be the corresponding preimages.
  If $y \in K$ is a limit point of $(x_k)_k$, then $[y] = p$, i.e., $y \in p$.
\end{proof}

We equip finite products of metric spaces with the maximum metric:
\begin{gather*}
  d(x,y) = \max_i \, d(x_i,y_i).
\end{gather*}
Thus, a finite product of isometric actions of $G$ gives rise to an isometric action.
In particular, we obtain natural continuous isometric actions $G \curvearrowright \widehat{G}_L^n$, and the quotient spaces $\widehat{G}_L^n \dslash G$ are well-defined up to uniform equivalence of the metric.

For $n = 1$, the quotient $\widehat{G}_L \dslash G$ is a singleton.
For $n = 2$, we have a canonical topological embedding $G \subseteq \widehat{G}_L^2 \dslash G$, identifying $g \in G$ with $[1,g] = [g^{-1},1] \in \widehat{G}_L^2 \dslash G$.
The uniform structure induced on $G$ from $\widehat{G}_L^2 \dslash G$ is called the \emph{Roelcke uniformity}.
Equivalently, the Roelcke uniformity is generated by a system of entourages $\bigl\{(g, v g u) : g \in G; \, v,u \in V\bigr\}$, where $V$ varies over neighbourhoods of the identity.
Consequently, $\widehat{G}_L^2 \dslash G$ is the \emph{Roelcke completion} of $G$, also denoted by $R(G)$.

For more details on these constructions, see, for example, \cite{BenYaacov-Tsankov:WAP}.

\begin{dfn}
  \label{dfn:RoelckePrecompact}
  Let $G$ be a topological group.
  \begin{itemize}
  \item A subset $A \subseteq G$ is \emph{Roelcke precompact} if it is totally bounded in the Roelcke uniformity.
    Equivalently, if the image of $A$ in the Roelcke completion $R(G)$ is precompact.
    Equivalently, if for every neighbourhood of the identity $U$, there exists a finite set $F \subseteq G$ such that $A \subseteq UFU$.
  \item The group $G$ is \emph{locally Roelcke precompact} if it admits a Roelcke precompact neighbourhood of the identity.
  \end{itemize}
\end{dfn}

We will also need some notions from coarse geometry of topological groups, as developed by Rosendal in \cite{Rosendal:CoarseGeometry}.

\begin{dfn}
  \label{dfn:GroupBoundedSet}
  Let $G$ be a topological group.
  A subset $B \subseteq G$ is \emph{coarsely bounded} if it is bounded for every continuous, left-invariant pseudo-metric on $G$.
  Equivalently, if for every sequence of symmetric open sets $U_n$ such that $G = \bigcup_n U_n$ and $U_n^2 \subseteq U_{n+1}$, we have $B \subseteq U_n$ for some $n$.
  The group $G$ is \emph{locally bounded} if it admits a coarsely bounded neighbourhood of the identity.
\end{dfn}

For locally Roelcke precompact groups, the coarsely bounded sets were identified by Zielinski~\cite{Zielinski:LocallyRoelckePrecompact}.

\begin{fct}
  \label{fct:GroupBoundedSetRoelckePrecompact}
  Let $G$ be a Polish, locally Roelcke precompact group.
  Then the coarsely bounded subsets of $G$ are precisely the Roelcke precompact subsets.
\end{fct}

\begin{dfn}
  \label{dfn:CoarselyProper}
  Let $G$ be a topological group acting by isometries on a metric space $(X, d)$.
  We say that the action $G \curvearrowright X$ is \emph{coarsely proper} if for every $r > 0$ and every (equivalently, some) $x \in X$, the set $\bigl\{ g \in G : d(gx , x) < r \bigr\}$ is coarsely bounded.

  A compatible, left-invariant metric $d_L$ on $G$ is called \emph{coarsely proper} if the left translation action $G \curvearrowright (G, d_L)$ is coarsely proper, i.e., the $d_L$-bounded sets are precisely the coarsely bounded subsets of $G$.
\end{dfn}

The following is an immediate consequence of \cite[Theorem~2.28]{Rosendal:CoarseGeometry}.

\begin{fct}
  \label{fct:CoarselyProperMetric}
  Let $G$ be a Polish group.
  Then $G$ is locally bounded if and only if admits a coarsely proper, compatible, left-invariant metric.
  In particular, if $G$ is locally Roelcke precompact, then it admits such a metric.
\end{fct}

We recall that a metric space is called \emph{proper} if all closed bounded sets are compact.

\begin{lem}
  \label{lem:lRpcActionProduct}
  Let $G$ be a Polish, locally Roelcke precompact group.
  Then the class of actions $G \curvearrowright X$ such that $X$ is a complete metric space, the action is continuous, isometric, and coarsely proper, and the quotient $X \dslash G$ is a proper metric space, is closed under finite products.
\end{lem}
\begin{proof}
  Let $G \curvearrowright X$ and $G \curvearrowright Y$ be as in the statement.
  Clearly, $X \times Y$ is a complete metric space, and the action $G \curvearrowright X \times Y$ is continuous, isometric, and coarsely proper.
  It remains to show that $(X \times Y) \dslash G$ is a proper metric space.
  In other words, we need to show that if $B \subseteq X \times Y$ is bounded, then its image in $(X \times Y) \dslash G$, which we denote by $B \dslash G$, is totally bounded.
  We may assume that $B = B_1 \times B_2$, where $B_1 \subseteq X$ and $B_2 \subseteq Y$ are bounded.

  Let $\delta > 0$ be given.
  By hypothesis, $B_1 \dslash G$ is totally bounded, so there exists a finite family $\set{x_i}_i \subseteq B_1$ such that for every $x \in B_1$ we have $d\bigl([x],[x_i]\bigr) < \delta$ for some $i$.
  In other words, $d(x,gx_i) < \delta$ for some $i$, and some $g \in G$.
  Notice that, in this case,
  \begin{gather*}
    g \in E_1 = \bigl\{ g \in G : d(B_1,gB_1) < \delta \bigr\}.
  \end{gather*}
  Since the action $G \curvearrowright X$ is coarsely proper and $B_1$ is bounded in $X$, the set $E_1$ is coarsely bounded.
  Similarly, there exists a finite family $\set{y_j}_j \subseteq B_2$ such that for every $y \in B_2$ we have $d(y,g y_j) < \delta$ for some $j$, and some $g$ in the coarsely bounded set
  \begin{gather*}
    E_2 = \bigl\{ g \in G : d(B_2,gB_2) < \delta \bigr\}.
  \end{gather*}

  Choose a neighbourhood of the identity $U \subseteq G$ such that $d(ux_i,x_i) < \delta$ and $d(uy_j,y_j) < \delta$ for all $u \in U$ and $i,j$.
  The set $E_1^{-1} E_2$ is coarsely bounded and therefore Roelcke precompact, so there exists a finite set $F$ such that $E_1^{-1} E_2 \subseteq U F U$.

  Let $(x,y) \in B = B_1 \times B_2$.
  By construction, there exist $i$, $j$, and $(g,h) \in E_1 \times E_2$ such that
  \begin{gather*}
    \max \bigl(d(x,g x_i), d(y,h y_j) \bigr) < \delta.
  \end{gather*}
  Then there exist $f \in F$ and $u,v \in U$ such that $g^{-1}h = ufv$.
  Now, letting $\sim$ mean ``at distance $<\delta$'', we have
  \begin{gather*}
    [x_i,f y_j] \sim [u^{-1}x_i,f v y_j] = [ g x_i, h y_j ] \sim [x,y].
  \end{gather*}
  It follows that the finite set of balls of radius $2\delta$ and centres at $[x_i,f y_j]$, where $i,j$, and $f$ vary, covers $B \dslash G$.
\end{proof}

\begin{thm}
  \label{thm:LocallyRoelckePrecompactProper}
  Let $X$ be a complete separable metric space and let $G \leq \Iso(X)$ be a closed subgroup.
  Assume that $X \dslash G$ is a proper metric space.
  Then the following are equivalent:
  \begin{enumerate}
  \item The group $G$ is locally Roelcke precompact and the action $G \curvearrowright X$ is coarsely proper.
  \item The quotient $X^n \dslash G$ is a proper metric space for every $n$.
  \item
    \label{item:LocallyRoelckePrecompactProperBoundedRpc}
    For every $x \in X$ and $r > 0$, the set $O = \bigl\{g \in G : d(x,gx) < r \bigr\}$ is Roelcke precompact.
  \end{enumerate}
\end{thm}
\begin{proof}
  \begin{cycprf}
  \item
    By \autoref{lem:lRpcActionProduct}.
  \item
    Let $V \subseteq G$ be a neighbourhood of the identity.
    Possibly replacing $V$ with a smaller set, we may assume that there exist $y \in X^k$, $\varepsilon > 0$ such that
    \begin{gather*}
      V = \bigl\{ g \in G : d(y,g y) < \varepsilon \bigr\}.
    \end{gather*}
    Define $\theta\colon G \rightarrow X^{2k} \dslash G$ as $\theta(g) = [y,gy]$.
    The set $\theta(O) \subseteq X^{2k} \dslash G$ is bounded, and therefore totally bounded.
    In particular, there exists a finite set $\set{h_i : i < m} \subseteq O$ such that for every $g \in O$ there exists $i < m$ with
    \begin{gather*}
      d\bigl( \theta(g), \theta(h_i) \bigr) < \varepsilon.
    \end{gather*}
    By definition, this means that there exists $v \in G$ such that
    \begin{gather*}
      \max  \bigl(d( y, v y ), d( g y, v h_i y ) \bigr)  < \varepsilon.
    \end{gather*}
    Equivalently, $v \in V$ and $u = h_i^{-1} v^{-1} g \in V$.
    In particular, there exist $v, u \in V$ such that $g = v h_i u$.
    Thus $O \subseteq V \{h_i : i < m\} V$, as promised.
  \item[\impfirst]
    Since a Roelcke precompact set is always coarsely bounded.
  \end{cycprf}
\end{proof}

The equivalence between the first and the last item of the following corollary is due to Zielinski~\cite{Zielinski:LocallyRoelckePrecompact}.

\begin{cor}
  \label{cor:LocallyRoelckePrecompactAction}
  Let $G$ be a Polish group.
  Then the following are equivalent:
  \begin{enumerate}
  \item The group $G$ is locally Roelcke precompact.
  \item The group $G$ admits a compatible, coarsely proper, left-invariant metric $d_L$, and the induced metric on $\widehat{G}_L^n \dslash G$ is proper for every $n$.
  \item The Roelcke completion $R(G)$ is locally compact.
  \end{enumerate}
\end{cor}
\begin{proof}
  \begin{cycprf}
  \item
    Let $d_L$ be a coarsely proper, compatible, left-invariant metric on $G$ (which exists by \autoref{fct:CoarselyProperMetric}), and let $\widehat{G}_L$ denote the completion.
    Then the action $G \curvearrowright \widehat{G}_L$ is coarsely proper, and $\widehat{G}_L \dslash G$ is a singleton.
    By \autoref{thm:LocallyRoelckePrecompactProper}, $\widehat{G}_L^n \dslash G$ is proper for every $n$.
  \item
    The Roelcke completion $R(G) = \widehat{G}_L^2 \dslash G$ is a proper metric space, and therefore locally compact.
  \item[\impfirst]
    Let $U \subseteq R(G)$ be any compact neighbourhood of $1 = [1_G,1_G]$.
    Then $U \cap G$ is a Roelcke precompact neighbourhood of $1$.
  \end{cycprf}
\end{proof}

It is obvious that if $G$ is a locally compact group and $G \actson X$ is a proper, isometric action, then the stabiliser of every point is compact.
We end this section with a suitable analogue of this fact in the Roelcke precompact setting.

\begin{lem}
  \label{lem:StabiliserOrbitEmbedding}
  Let $X$ be a complete metric space, let $G$ be a Polish group, and let $G \curvearrowright X$ continuously by isometries.
  For $x \in X$, let $G_x \leq G$ denote the stabiliser of $x$ for this action.
  For any fixed $n$, define a map $\iota_x\colon X^n \dslash G_x \rightarrow X^{n+1} \dslash G$ sending $[y] \mapsto [x,y]$.
  Then:
  \begin{enumerate}
  \item The map $\iota_x$ is well-defined.
    It is $1$-Lipschitz, and sends unbounded sets to unbounded sets.
  \item Suppose that $Gx$ is a $G_\delta$ set.
    Then $\iota_x$ is a homeomorphic embedding and $X^n \dslash G$ is a proper metric space if and only if the image of $\iota_x$ is a proper metric space.
  \item If $G \curvearrowright X$ is transitive, then $\iota_x$ is surjective.
  \end{enumerate}
\end{lem}
\begin{proof}
  Let $y,z \in X^n$, and let $\overline{x} = (x,x,\ldots,x) \in X^n$ and $g \in G$.
  Then, on the one hand,
  \begin{gather*}
    d(G_x y, G_x z) = d\bigl( G_x \cdot (x,y) , G_x \cdot (x,z) \bigr) \geq d\bigl( G \cdot (x,y) , G \cdot (x,z) \bigr),
  \end{gather*}
  so $\iota_x$ is indeed well-defined and $1$-Lipschitz.
  On the other hand,
  \begin{gather*}
    d( G_x y, G_x \overline{x})
    =
    d(y, \overline{x})
    \leq
    d(x,gx) +  d(y, g\overline{x})
    \leq
    2 d\bigl( (x,y) , g \cdot (x,\overline{x})\bigr),
  \end{gather*}
  so $\iota_x$ sends unbounded sets to unbounded sets.
  For the second assertion, let $y \in X^n$ and $\varepsilon > 0$.
  Let $V = \bigl\{ g \in G : d(y,gy) < \varepsilon \bigr\}$.
  By Effros's theorem \cite{Effros:TransformationGroups}, $Vx$ is open in $Gx$, so it contains $B(x,\delta) \cap Gx$ for some $\delta > 0$.
  Suppose now that $d\bigl( [x, y], [x, z] \bigr) < \delta$ for some $z \in X^n$.
  Then there exists $g \in G$ such that $d(x, g x) < \delta$ and $d(y, g z) < \delta$.
  Consequently, there exists $h \in V$ such that $hgx = x$.
  Then $hg \in G_x$, and
  \begin{gather*}
    d\bigl( y, hg z\bigr)
    \leq d\bigl( y, hy \bigr) + d(hy, hgz)
    < \varepsilon + \delta.
  \end{gather*}
  Therefore $d\bigl( [y], [z] \bigr) < \varepsilon + \delta$, which is enough to show that $\iota_x^{-1}$ is continuous on the image of $\iota_x$.
  By the first item, a set $K \subseteq X^n \dslash G_x$ is bounded if and only if its image is, so properness is preserved as well.

  The last assertion is immediate.
\end{proof}

\begin{prp}
  \label{prp:lRpcStabiliser}
  Let $X$ be a complete, separable metric space, and let $G \leq \Iso(X)$ be a closed subgroup that acts transitively on $X$.
  Then the following are equivalent:
  \begin{enumerate}
  \item The group $G$ is locally Roelcke precompact and the action $G \curvearrowright X$ is coarsely proper.
  \item For some (every) $x \in X$, the stabiliser $G_x$ is Roelcke precompact and $X \dslash G_x$ is a proper metric space.
  \end{enumerate}
\end{prp}
\begin{proof}
  \begin{cycprf}
  \item
    By \autoref{thm:LocallyRoelckePrecompactProper}, $X^{n+1} \dslash G$ is a proper metric space for every $n$, and by \autoref{lem:StabiliserOrbitEmbedding}, $X^n \dslash G_x$ is proper.
    By \autoref{thm:LocallyRoelckePrecompactProper} applied to $G_x \curvearrowright X$, the stabiliser $G_x = \{g \in G_x : d(x,gx) < 1 \}$ is Roelcke precompact.
  \item[\impfirst]
    Since $G_x$ is Roelcke precompact, it is coarsely bounded, so $G_x \curvearrowright X$ is coarsely proper.
    By \autoref{thm:LocallyRoelckePrecompactProper}, $X^n \dslash G_x$ is proper for every $n$, and by \autoref{lem:StabiliserOrbitEmbedding}, $X^{n+1} \dslash G$ is proper.
    We conclude by \autoref{thm:LocallyRoelckePrecompactProper}.
  \end{cycprf}
\end{proof}


\section{A metrisation theorem and prime models}
\label{sec:DefinableMetrisation}

Ordinarily, predicates and formulas in continuous logic are required to be bounded, taking values in some fixed compact interval.
However, coarse geometric structure, which plays a key role in this paper, is naturally expressed by an unbounded definable metric.
Moreover, most examples, such as the Urysohn space, come already equipped such a metric.

In order to incorporate an unbounded metric in the logic, we will allow predicates (and formulas) to take values in the compact interval $[-\infty, \infty]$.
This has two consequences: first, the uniform continuity requirement for predicates should be considered with respect to the uniformity of this interval as a compact space; and second, by the compactness theorem, in some models the predicates can actually take infinite values.
In particular, if a definable metric is unbounded, then it should be viewed as a definable \emph{generalised metric}, i.e., a metric which can take the value $\infty$.

One should not worry too much about this.
The quantifiers $\sup$ and $\inf$ work just as well on $[-\infty,\infty]$, which is order-homeomorphic to $[0,1]$.
As for connectives, a $[-\infty, \infty]$-valued formula can be combined with other formulas using any continuous function (continuous connective) defined on $[-\infty, \infty]$, keeping in mind that the ordinary arithmetic operations do not always admit a continuous extension to $[-\infty, \infty]$.

Alternatively, let $\theta\colon [0,\infty] \rightarrow [0,1]$ be the order-homeomorphism $\theta(t) = 1- e^{-t}$.
It is not difficult to check that $d$ is a generalised metric if and only if $d' = \theta \circ d$ is a $[0,1]$-valued metric that satisfies the stronger triangle inequality
\begin{gather}
  \label{eq:GeneralisedDistance01}
  d'(x,z) \leq d'(x,y) + d'(y,z) - d'(x,y) d'(y,z).
\end{gather}
Therefore, specifying a definable generalised metric is the same thing as specifying a definable metric that satisfies \autoref{eq:GeneralisedDistance01}.
Similarly, we may replace other unbounded predicates by bounded ones if desired.

\begin{conv}
  \label{conv:DefinableMetric}
  A \emph{definable metric} on a structure $M$ is exactly that: a metric defined by a formula.
  A definable metric may be unbounded, in which case the defining formula has to be $[0,\infty]$-valued, in the sense of the preceding discussion.
  Variants such as bounded or generalised metric will be qualified as such explicitly.

  On the class of all models of a theory $T$ we can either have a generalised definable metric or a bounded one, and we shall attempt to be careful to always qualify them as one or the other.
\end{conv}

\begin{dfn}
  \label{dfn:OpenEquivalenceRelation}
  Let $T$ be a theory.
  An \emph{invariant binary relation} on (models of) $T$ is a set of $2$-types $R \subseteq \tS_2(T)$.
  It induces a binary relation on each model $M \vDash T$, where $R(a,b)$ holds if $\tp(a,b) \in R$.
  If it induces an equivalence relation on each model, then it is an \emph{invariant equivalence relation}.
  If $R \subseteq \tS_2(T)$ is an open subset, then it is an \emph{open relation}.
\end{dfn}

Clearly, if $d$ is a definable generalised metric on models of $T$, then the relation $d(a,b) < \infty$ is an open equivalence relation.
Below we prove the converse, using an argument inspired by Urysohn.

\begin{dfn}
  Let $T$ be a theory.
  Given subsets $A,B \subseteq \tS_2(T)$, define
  \begin{align*}
    A \cdot B &= \bigl\{\tp(a, c) : \exists b \ \tp(a, b) \in A, \, \tp(b, c) \in B\bigr\} \subseteq \tS_2(T), \\
    A^\op &= \bigl\{ \tp(b,a) : \tp(a,b) \in A \bigr\}.
  \end{align*}
  We say that $A$ is \emph{symmetric} if $A = A^\op$.
\end{dfn}

By amalgamation of types, this operation is associative, so an iterated product $\prod_{i<n} A_i$ depends only on the sequence of sets $(A_i : i < n)$.
Moreover, if $A$ and $B$ are compact, then $A \cdot B$ is compact.
Finally, this product operation admits a neutral element, namely the diagonal
\begin{gather*}
  \Delta = \bigl\{ p(x, y) \in \tS_2(T) : p \vDash x = y  \bigr\}.
\end{gather*}

\begin{lem}
  \label{lem:DefinableMetrisationBase}
  Let $T$ be a theory.
  Let $K_i \subseteq \tS_2(T)$ be compact for $i < n$, let $U \subseteq \tS_2(T)$ be open, and assume that $\prod_i K_i \subseteq U$.
  Then there exist open sets $V_i \subseteq \tS_2(T)$ such that $K_i \subseteq V_i$ and $\prod_i \overline{V}_i \subseteq U$.
\end{lem}
\begin{proof}
  By induction, it suffices to consider the case $n = 2$.
  Suppose, towards contradiction, that for every pair $\alpha = (V_0, V_1)$ of open neighbourhoods of $K_0$ and $K_1$, there exists a type $p_\alpha(x,y,z) \in \tS_3(T)$ such that $p_\alpha|_{xy} \in \cl{V_0}$, $p_\alpha|_{yz} \in \cl{V_1}$, and $p_\alpha|_{xz} \notin U$.
  The $p_\alpha$ form a net if we order the indices by reverse inclusion and any limit point of this net witnesses that $K_0 \cdot K_1 \nsubseteq U$.
\end{proof}

The following is a technical definition used in the proof of the metrisation theorem.
\begin{dfn}
  \label{dfn:DefinableMetrisationGoodFamily}
  Let $T$ be a theory, and let $\sim$ be an open equivalence relation on models of $T$.
  Let $S \sub \bR^{>0}$ and for each $s \in S$, let $U_s \subseteq \tS_2(T)$ be an open set.

  We call the family $U_S = \{U_s : s \in S\}$ \emph{good} if
  \begin{itemize}
  \item each $U_s$ is symmetric and satisfies $\Delta \subseteq U_s \subseteq \overline{U_s} \subseteq {\sim}$; and
  \item whenever $s_0, \dots, s_{n-1}, t \in S$ are such that $\sum_{i<n} s_i < t$, we have $\prod_{i<n} \overline{U_{s_i}} \subseteq U_t$.
  \end{itemize}
\end{dfn}

\begin{lem}
  \label{lem:DefinableMetrisationGoodFamilyExtension}
  Let $S \sub \bR^{>0}$ be finite and let $U_S$ be a good family as per \autoref{dfn:DefinableMetrisationGoodFamily}.
  Then for every $t > 0$, $U_S$ can be extended to a good family $U_{S \cup \{t\}}$.

  Moreover, if $V_0 \subseteq \tS_2(T)$ is open such that $\Delta \subseteq V_0 \subseteq \overline{V_0} \subseteq {\sim}$, and if $t < s$ (respectively, $t > s$) for all $s \in S$, then we may choose $U_t$ to be contained in (respectively, contain) $V_0$.
\end{lem}
\begin{proof}
  Let
  \begin{equation*}
    K = \bigcup \set[\Big]{\prod_{i<n} \overline{U_{s_i}} : s_0, \dots, s_{n-1} \in S,\ \sum_{i < n} s_i < t}.
  \end{equation*}
  Note that $K$ is compact as a finite union of compact sets.
  As ${\sim}$ is an equivalence relation and each $\cl{U_{s_i}} \sub {\sim}$, we have that $K \sub {\sim}$.
  We include the case $n = 0$ in the above union, so $\Delta = \prod \emptyset \sub K$.
  Since each $U_s$ is symmetric, so is $K$.

  Set $K_s = K'_s = \overline{U_s}$ for $s \in S$ and $K_t = K$.
  Let us consider a sequence $(s_i : i < n)$, where $s_i \in S \cup \{t\}$, and $r \in S$, such that $\sum_i s_i < r$.
  By the definition of $K$ and the fact that the family $U_S$ is good, we have $\prod_i K_{s_i} \subseteq U_r$.
  By \autoref{lem:DefinableMetrisationBase}, there exists an open set $V \supseteq K$ such that, if we define $K'_t = \overline{V}$, then $\prod_i K'_{s_i} \subseteq U_r$ as well.
  There exist only finitely many such sequences to consider, so there exists $V$ that works for all of them.
In addition, we may assume $\overline{V} \subseteq {\sim}$.
  Letting $U_t = V \cap V^\op$, the family $U_{S \cup \{t\}}$ is good.

  The moreover part is immediate.
\end{proof}

\begin{thm}
  \label{thm:DefinableMetrisation}
  Let $T$ be a theory in a separable language, and let $\sim$ be an open equivalence relation on models of $T$.
  Then $T$ admits a definable generalised metric $d$ such that $d(a,b) < \infty$ if and only if $a \sim b$.

  Moreover, this definable metric is unique up to uniform and coarse equivalence.
  In other words, if $d$ and $d'$ are two such metrics, then for every $\varepsilon > 0$ there exists $\delta > 0$ such that
  \begin{gather*}
    d(x,y) < \delta \quad \Longrightarrow \quad d'(x,y) < \varepsilon,
    \qquad
    d(x,y) > 1/\delta \quad \Longrightarrow \quad d'(x,y) > 1/\varepsilon,
  \end{gather*}
  and similarly the other way round.
\end{thm}
\begin{proof}
  Choose a countable dense set $S \sub \bR^{> 0}$.
  By \autoref{lem:DefinableMetrisationGoodFamilyExtension} and induction, we may construct a good family of open sets $U_S$.
  Since the language is separable, $\tS_2(T)$ is second-countable, so we can ensure that $\bigcap_s U_s = \Delta$ and $\bigcup_s U_s = {\sim}$.
  Since $S$ is dense and for all $s < t$, $\cl{U_s} \sub U_t$, for every $p \in \tS_2(T)$, we have the equality
  \begin{gather*}
    \sup \, \bigl\{ s \in S : p \notin \cl{U_s} \bigr\}
    = \inf \, \bigl\{ s \in S : p \in U_s \bigr\}.
  \end{gather*}
  Call this common value $d(p)$.
  In particular, $p \in \Delta$ if and only if $d(p) = \inf S = 0$, and $p \notin {\sim}$ if and only if $d(p) = \sup S = \infty$.

  For $p \in \tS_2(T)$ and $t \in \bR^+$, $d(p) < t \iff p \in \bigcup_{s < t} U_s$ and $d(p) \leq t \iff p \in \bigcap_{s > t} \cl{U_s}$, so $d$ is continuous.
  Therefore we may identify it with a $[0,\infty]$-valued formula $d(x,y)$.
  Consider $a$, $b$, and $c$ in some model of $T$.
  By construction, $d(a,b) = d(b,a)$, and $d(a,b) = 0$ if and only if $\tp(a,b) \in \Delta$, if and only if $a = b$.
  Assume now that $s,t,r \in S$, $d(a,b) < s$, $d(b,c) < t$, and $s+t < r$.
  Then $\tp(a,b) \in U_s$, $\tp(b,c) \in U_t$, and $\overline{U_s} \cdot \overline{U_t} \subseteq U_r$.
  Therefore $d(a,c) \leq r$, which is enough to prove the triangle inequality.
  We conclude that the formula $d$ defines a generalised metric.
  In addition, $d(a,b) < \infty$ if and only if $\tp(a,b) \in {\sim}$ if and only if $a \sim b$, as desired.

  For the moreover part, assume that $d'$ is another such generalised metric.
  If $\varepsilon > 0$, then the closed condition $d(x,y) = \infty$ implies the open condition $d'(x,y) > 1/\varepsilon$.
  Therefore, by compactness, there exists $\delta > 0$ such that $d(x,y) > 1/\delta$ implies $d'(x,y) > 1/\varepsilon$.
  The argument for uniform equivalence near zero is similar (and standard).
\end{proof}

\begin{rmk}
  \label{rmk:metrisation-Robinson}
  The proof of \autoref{thm:DefinableMetrisation} only uses compactness and amalgamation, so the conclusion holds in the greater generality of \emph{Robinson theories}, i.e., universal first-order theories whose models satisfy the amalgamation property.
In that case, both the equivalence relation $\sim$ and the generalised metric $d$ are defined using quantifier-free types.
\end{rmk}

Let us turn to prime models.
Let $T$ be a complete theory in a  separable language.
Recall that $p \in \tS_n(T)$ is \emph{isolated} if the logic topology and the topology given by the type metric coincide at $p$; equivalently, if $p$ is realised in every model of $T$; equivalently, if the set of realisations of $p$, in any model of $T$, is a definable set.
Let $\fI_n(T) \subseteq \tS_n(T)$ denote the space of isolated $n$-types of $T$.
A model $M \vDash T$ is \emph{atomic} if it only realises isolated types.

It is a standard fact that a model $M \vDash T$ is \emph{prime} (embeds elementarily into every model of $T$) if and only if it is separable and atomic.
A prime model exists if and only if $\fI_n(T)$ is dense in $\tS_n(T)$ for every $n$, and it is unique up to isomorphism.
Moreover, if $M$ is a prime model, then it is \emph{homogeneous}: if $a,b \in M^n$ have the same type, then there exists $g \in \Aut(M)$ such that $d(ga,b)$ is arbitrarily small.

Given $M \vDash T$ and $p \in \tS_n(T)$, let $p(M)$ denote the set of realisations of $p$ in $M$.
For $a \in M^n$, the type distance $d\bigl( \tp(a), p \bigr)$ is equal to the set distance $d\bigl( a, p(N) \bigr)$ for a sufficiently saturated $N \succeq M$, but in general may be strictly smaller than $d\bigl( a, p(M) \bigr)$ (in fact, $p$ need not even be realised in $M$).
It is an immediate yet important observation that this cannot happen when $p$ is isolated.

\begin{fct}
  \label{fct:IsolatedTypeDistance}
  Let $T$ be a complete theory and $p \in \fI_n(T)$.
  Then the distance to $p$ in $\tS_n(T)$ is a continuous function, i.e., a definable predicate, that we may denote by $d(x,p)$.
  If $M \vDash T$ is any model, and $a \in M^n$, then $d(a,p)$ is the distance from $a$ to the set of realisations of $p$ in $M$.

  Consequently, if $M$ is a prime model of $T$ and $G = \Aut(M)$, then $M^n \dslash G \simeq \fI_n(T)$ isometrically via the map $[a] \mapsto \tp(a)$, and $[a]$ is the set of realisations of $\tp(a)$ in $M$.
\end{fct}
\begin{proof}
  Since the set of realisations of $p$ is definable, we may quantify over it.
  In particular, there exists a formula that we may denote by $\varphi(x) = \inf_{p(y)} d(x,y)$, such that $\varphi(a)$ evaluates to the distance to $p(M)$ in every model $M \vDash T$.
  In particular, this distance cannot change when passing to an elementary extension, so it coincides with $d\bigl( \tp(a), p \bigr)$.
  This proves our first assertion.

  Assume now that $M$ is prime, and let $a \in M^n$.
  Then every $b \in G a$ realises $\tp(a)$, and since formulas are uniformly continuous, every $b \in \overline{G a} = [a]$ realises $\tp(a)$.
  Conversely, since $M$ is homogeneous, every realisation of $p$ in $M$ belongs to $[a]$.
  Since the types realised in $M$ are exactly the isolated ones, the map $[a] \mapsto \tp(a)$ defines a bijection $M^n \dslash G \simeq \fI_n(T)$.
  It is isometric by our first assertion.
\end{proof}

We conclude this section with a construction of a structure from an appropriate group action $G \curvearrowright A$.
The constructed structure is automatically the prime model of its theory.

\begin{dfn}
  \label{dfn:GroupActionStructure}
  Let $A$ be a complete separable metric space, $G$ a group, and $G \curvearrowright A$ an action by isometries.
  We define a metric structure $A^G$ by equipping $A$ with an $n$-ary predicate symbol $D_a$ for each $n \in \bN$ and each $a \in A^n$, interpreted as
  \begin{gather*}
    D_a(b) = d(b, Ga).
  \end{gather*}
\end{dfn}

If $a_k \rightarrow a$ in $A^n$, then $D_{a_k} \rightarrow D_a$ uniformly modulo $\Th(A^G)$.
Therefore, the language of $A^G$ is separable (modulo its theory).

In $A^G$, a predicate $D_a$ only depends on the orbit closure $[a] = \overline{G a} \in A^n \dslash G$.
Therefore, given $p \in A^n \dslash G$, we may denote by $D_p$ any predicate $D_a$ such that $p = [a]$.
Clearly, $G$ acts on $A^G$ by automorphisms.

\begin{prp}
  \label{prp:GroupActionStructureMinimal}
  Let $(A,d)$ be a complete separable metric space, let $G \curvearrowright A$ be an action by isometries, and let $A^G$ be as in \autoref{dfn:GroupActionStructure}.
  Then $A^G$ is the prime model of its theory $T$.
  If $G \leq \Iso(A)$ is a closed subgroup, then $\Aut(A^G) = G$.

  Conversely, let $\cL$ be some vocabulary, and $A^\cL$ an interpretation of $\cL$ on $(A,d)$ such that $\cL$ is separable modulo $T^\cL = \Th(A^\cL)$.
  Suppose moreover that $A^\cL$ is the prime model of $T^\cL$ and $G = \Aut(A^\cL)$.
  Then $D_a$ is definable in $A^\cL$ for every $a \in A^n$.
\end{prp}
\begin{proof}
  Let $a \in A^n$.
  Since $G$ acts on $A^G$ by automorphisms, every $b \in G a$ realises $\tp(a)$, and since formulas are uniformly continuous, so does every $b \in [a]$.
  Conversely, if $b$ realises $\tp(a)$, then $D_a(b) = D_a(a) = 0$, so $b \in [a]$.
  Therefore, if $p = \tp(a)$, then the set of realisations of $p$ in $A^G$ is $[a]$, and the distance to it is defined by $D_a$.
  In particular, every type realised in $A^G$ is isolated, so $A^G$ is the prime model of its theory.
  If $G \leq \Iso(A)$, then $G$ is, by construction, dense in $\Aut(A^G)$; if $G$ is closed, then $G = \Aut(A^G)$.

  Conversely, assume that $A^\cL$ has separable language, is prime, and its automorphism group is $G$.
  If $a \in A^n$ and $p = \tp^\cL(a)$, then the set of realisations of $p$ is $[a]$, and $D_a(x) = d(x,p)$ is definable in $A^\cL$ by \autoref{fct:IsolatedTypeDistance}.
\end{proof}

The second part of \autoref{prp:GroupActionStructureMinimal} asserts that $A^G$ is the \emph{least} prime structure on $A$ with automorphism group $G$.

\begin{rmk}
  \label{rmk:GroupActionStructureAlternateLanguage}
  Consider the special case of \autoref{dfn:GroupActionStructure}, where $G$ is a Polish group equipped with a left-invariant metric $d_L$, and the action is $G \curvearrowright \widehat{G}_L$.

  By \autoref{fct:GroupActionCompletion}, this extends to a semi-group action $\widehat{G}_L \curvearrowright \widehat{G}_L$.
  For $a \in \widehat{G}_L$, let us denote the map $x \mapsto xa$ by $R_{a} \colon \widehat{G}_L \rightarrow \widehat{G}_L$.
  These maps are uniformly continuous, and the family of symbols consisting of the distance and the maps $R_a$ for $a \in \widehat{G}_L$ (or even just in $G$) can be construed as an alternate language for $\widehat{G}_L$.
  It will be useful, especially for the examples, to observe that this language is bi-definable with the language of $\widehat{G}_L^G$ given in \autoref{dfn:GroupActionStructure}.

  Indeed, in one direction, for $g \in G^n$,
  \begin{gather*}
    D_g(x) = \inf_y \, \bigvee_{i<n} d(x_i,y g_i).
  \end{gather*}
  Since $G$ is dense in $\widehat{G}_L$, this is enough to recover the language of $\widehat{G}_L^G$.

  Conversely, fix $a \in \widehat{G}_L$.
  Since $R_a$ is uniformly continuous, there exists a continuous, increasing continuity modulus $\mu\colon [0,\infty] \rightarrow [0,\infty]$ such that $\mu(0) = 0$ and
  \begin{gather*}
    d(xa,ya) \leq \mu \circ d(x,y).
  \end{gather*}
  We may assume that $\mu(t) \geq t$ for all $t$.
  If $x,y \in \widehat{G}_L$ and $h \in G$, then
  \begin{gather*}
    \mu \circ d\bigl( (x,y), (h,ha)\bigr)
    \geq \mu \circ d(x, h) \vee d(y,ha)
    \geq d(xa, ha) \vee d(y,ha)
    \geq d(xa,y) / 2.
  \end{gather*}
  Therefore, in $\widehat{G}_L^G$,
  \begin{gather*}
    \mu \circ D_{[1,a]}(x,y)
    = \mu \circ d\bigl( [x,y], [1,a]\bigr)
    \geq d(xa,y) / 2.
  \end{gather*}
  Consequently,
  \begin{gather*}
    d(xa,y) = \inf_z \, \bigl( d(y,z) + 2 \mu \circ D_{[1,a]}(x,z) \bigr),
  \end{gather*}
  where the infimum is attained at $z = xa$.
  We conclude that the map $R_{a}$ is definable in $\widehat{G}_L^G$ for every $a \in \widehat{G}_L$.
\end{rmk}


\section{Locally $\aleph_0$-categorical theories}
\label{sec:LocalA0theories}

The main goal of the present paper is to extend some of the theory of $\aleph_0$-categorical structures to ones that are \emph{locally} so.
While we are quite confident in our definition of locally $\aleph_0$-categorical theories, general local theories are more subtle and several definitions are possible.
For the purposes of this paper, the following weak variant will suffice.

\begin{dfn}
  \label{dfn:WeaklyLocalTheory}
  Let $T$ be a theory and let $\sim$ be an equivalence relation on models of $T$.
  We say that $T$ is \emph{weakly local} relative to $\sim$, or that the pair $(T,\sim)$ is a \emph{weakly local theory}, if the following hold:
  \begin{itemize}
  \item If $M \vDash T$ and $M_0 \subseteq M$ is an equivalence class, then $M_0 \preceq M$.
  \item A map between models of $T$ is elementary if and only if it is elementary on each equivalence class and sends distinct equivalence classes into distinct ones.
  \end{itemize}
  The equivalence classes in $M$ are its \emph{local components}.
  If $M \vDash T$ has a single local component, then $M$ is a \emph{local model} of $T$.

  A \emph{localising metric} for a local theory $(T,{\sim})$ is a definable generalised metric $d$ such that $a \sim b$ if and only if $d(a,b) < \infty$.
\end{dfn}

By \autoref{thm:DefinableMetrisation}, if $(T, \sim)$ is a weakly local theory in a separable language and $\sim$ is open, then it  admits a localising metric, which is unique up to uniform and coarse equivalence.

\begin{dfn}
  \label{dfn:LocallyA0Categorical-prelim}
  A weakly local theory $(T, \sim)$ is \emph{locally $\aleph_0$-categorical} if it has a separable language and it admits a unique local separable model up to isomorphism.
\end{dfn}

\begin{dfn}
  \label{dfn:LocallyIsolated}
  Let $T$ be a complete theory, and let $p = \tp(a) \in \tS_n(T)$.
  For $i,j < n$, we write $i \sim_p j$ if $\tp(a_i,a_j) \in \tS_2(T)$ is isolated.
  Say that $p$ is \emph{locally isolated} if the following conditions hold:
  \begin{itemize}
  \item $\sim_p$ is an equivalence relation on $n$;
  \item for each equivalence class $I \subseteq n$, $p_I = \tp(a_I)$ is isolated, where $a_I = (a_i : i \in I)$; and
  \item if $q \in \tS_n(T)$ is such that $\sim_q$ coincides with $\sim_p$, and $q_I = p_I$ for each equivalence class $I$, then $q = p$.
  \end{itemize}
\end{dfn}

\begin{lem}
  \label{lem:LocallyIsolatedStep0}
  Assume that $T$ is complete and all types of $T$ are locally isolated.
  Then every $1$-type is isolated.
  Moreover, for every $1$-type $p$, there exists $\varepsilon > 0$ such that for every $a,b$ that realise $p$, if $d(a,b) < \varepsilon$, then $\tp(a,b)$ is isolated.
\end{lem}
\begin{proof}
  A locally isolated $1$-type must be isolated by definition.
  Moreover, local isolation implies that for each $1$-type $p$, there exists at most one non-isolated $2$-type $q = \tp(a,b)$ such that $\tp(a) = \tp(b) = p$.
  Since $p$ is isolated, so is $\tp(a,a)$.
  Therefore, if such a type $q$ exists, then $a \neq b$, and we may choose $\varepsilon = d(a,b) > 0$; and if no such type exists, then any $\varepsilon > 0$ will do.
\end{proof}

\begin{lem}
  \label{lem:LocallyIsolatedStep1}
  Assume that $T$ is a complete theory in a separable language and that all types of $T$ are locally isolated.
  Let $q(x) \in \tS_n(T)$ be isolated, and let $Q \subseteq \tS_{n+1}(T)$ be the collection of types extending $q$.
  Then the set of isolated types in $Q$ is dense in $Q$.
\end{lem}
\begin{proof}
  Let $U \subseteq \tS_{n+1}(T)$ be open, and assume that $p(x,y) \in U \cap Q$.
  Let $r(y) \in \tS_1(T)$ be the restriction of $p$ to the last variable.

  If $p$ is isolated, then we are done, so assume that it is not.
  Since the type space is metrisable, there exists a sequence of types $p_k(x,y) \in \tS_{n+1}(T)$ such that $p_k \rightarrow p$ in the logic topology but not in the type metric.
  Let $q_k(x)$ and $r_k(y)$ be their respective restrictions.
  Then $q_k \rightarrow q$ and $r_k \rightarrow r$ in the logic topology.
  The type $q$ is isolated by hypothesis, and $r$ is automatically isolated as a $1$-type, so $q_k \rightarrow q$ and $r_k \rightarrow r$ in the type metric as well.
  This means that (in a sufficiently saturated model) we may find realisations $(a_k,b_k) \vDash p_k$, as well as $a'_k \vDash q$ and $b'_k \vDash r$, such that $d(a_kb_k, a'_kb'_k) \rightarrow 0$.
  Let $p'_k = \tp(a'_k,b'_k)$.
  Then $d(p_k,p'_k) \rightarrow 0$, and since $p_k \rightarrow p$ in the logic topology but not in the metric, the same holds for the sequence $(p'_k)_k$.
  Therefore, for some large enough $k$, we have $p'_k \in U$ and $p'_k \neq p$.
  However, $p'_k$ and $p$ have the same restrictions to $x$ and to $y$, both of which are isolated.
  Since $p$ is non-isolated, by local isolation, $p'_k$ must be isolated (or else $p = p'_k$), showing that $U \cap Q$ contains an isolated type.
\end{proof}

\begin{lem}
  \label{lem:LocallyIsolatedStep2}
  Assume that $T$ is a complete theory in a separable language and that all types of $T$ are locally isolated.
  For $M \vDash T$ and $a,b \in M$, say that $a \sim b$ if $\tp(a,b)$ is isolated.
  Then $\sim$ is an invariant equivalence relation, and $T$ is weakly local relative to $\sim$.
\end{lem}
\begin{proof}
  Clearly, $\sim$ is an invariant, symmetric relation.
  Reflexivity follows from \autoref{lem:LocallyIsolatedStep0}.
  Transitivity follows from the local isolation of types.

  Assume now that $M \vDash T$, and $N \subseteq M$ is an equivalence class.
  Since isolated types form a metrically closed set, $N$ is a closed subset of $M$.
  Therefore, in order to see that $N \preceq M$, all we need to show is that if $a \in N^x$, $b \in M$ and $\varphi(x, y)$ is a formula with $y$ a singleton such that $\varphi(a,b) > 0$, then there exists $b' \in N$ such that $\varphi(a,b') > 0$.
  By local isolation, every type realised in $N$ is isolated, and in particular, so is $q = \tp(a)$.
  By \autoref{lem:LocallyIsolatedStep1}, there an isolated type $p \in \tS_{x,y}(T)$ such that $\varphi(p) > 0$ and whose restriction to $x$ is $q$.
  Since $p$ is isolated, we may quantify over the set of its realisations, and since its restriction to $x$ is $q$, we have:
  \begin{gather*}
    M \vDash \inf_{(x,y) \vDash p} d(a,x) = 0.
  \end{gather*}
  In other words, there exist realisations $(a_m,b_m) \in M^{x,y}$ of $p$ such that $a_m \rightarrow a$.
  In particular, $\varphi(a_m,b_m) = \varphi(p) > 0$.
  We may assume that $a$ is non-empty.
  Denote by $c$ the first member of $a$, and by $c_m$ the first member of $a_m$.
  If $d(a,a_m)$ is small enough, then, first, $\varphi(a,b_m) > 0$, and second, $\tp(c,c_m)$ is isolated by \autoref{lem:LocallyIsolatedStep0}.
  By local isolation, $\tp(a,a_m,b_m)$ is isolated, so $b_m \in N$.
  This concludes the proof that $N \preceq M$.

  The condition regarding elementary maps now follows immediately from local isolation, so $(T,\sim)$ is indeed weakly local.
\end{proof}

\begin{lem}
  \label{lem:LocallyIsolatedStep3}
  Assume that $T$ is a complete theory in a separable language and that all types of $T$ are locally isolated.
  Then for every $n$, the set of isolated $n$-types is open in $\tS_n(T)$.
  In particular, the joint isolation relation $\sim$ of \autoref{lem:LocallyIsolatedStep2} is open.
\end{lem}
\begin{proof}
  Assume that $p_k \rightarrow p$ in $\tS_2(T)$, where $p_k$ are non-isolated.
  Let $q(x)$ and $r(y)$ be the restrictions of $p(x,y)$ to the first and second variables, respectively, and similarly for $q_k$ and $r_k$.
  Then $q_k \rightarrow q$ and $r_k \rightarrow r$.
  Since $q$ and $r$ are $1$-types, they are isolated, so $d(q_k,q) + d(r_k,r) \rightarrow 0$.
  By passing to a subsequence, we may assume that $d(q_k,q) + d(r_k,r) < 2^{-k-2}$.
  In particular, $d(q_k,q_{k+1}) + d(r_k,r_{k+1}) < 2^{-k}$.

  By \autoref{lem:LocallyIsolatedStep2}, $T$ is a weakly local theory.
  We assume that $T$ admits non-isolated types, so it must admit a non-local model $M \vDash T$.
  Let $M'$ be any local component of $M$.
  Then $M'$ is a model of $T$, and since the types $q_k$ are isolated, we can construct a sequence $(a_k)_k$ in $M'$, such that $a_k \vDash q_k$ and $d(a_k,a_{k+1}) < 2^{-k}$.
  Let $a = \lim_k a_k \in M'$.
  Similarly, in another local component, construct a sequence $(b_k)_k$ such that $b_k \vDash r_k$ and $b_k \rightarrow b$.
  By local isolation, $\tp(a_k,b_k) = p_k$, so $p = \tp(a,b)$.
  Since $a$ and $b$ belong to distinct local components, $p$ is non-isolated.

  This proves that the set of isolated types in $\tS_2(T)$ is open.
  A type in $\tS_n(T)$ is isolated if and only if all its restrictions to $2$-types are isolated, by local isolation, so the isolated types form an open set in $\tS_n(T)$ as well.
\end{proof}

\begin{thm}[Local Ryll-Nardzewski]
  \label{thm:LocalRyllNardzewski}
  Let $T$ be a theory in a separable language, and let $\sim$ be an invariant equivalence relation on models of $T$.
  Then the following are equivalent:
  \begin{enumerate}
  \item The theory $(T, \sim)$ is locally $\aleph_0$-categorical.
  \item The theory $T$ is complete, the types of $T$ are locally isolated, and $a \sim b$ if and only if $\tp(a,b)$ is isolated.
  \end{enumerate}
  If these equivalent conditions hold, then, in addition:
  \begin{itemize}
  \item The isolated types in $\tS_n(T)$ form a dense open set for each $n$.
  \item A model of $T$ is local if and only if it is atomic; in particular, the unique local separable model of $T$ is its prime model.
  \item The isomorphism type of any model of $T$ is determined by the isomorphism types of its local components and their multiplicity.
  \end{itemize}
\end{thm}
\begin{proof}
  Assume first that $(T,\sim)$ is locally $\aleph_0$-categorical, and let $M_0$ denote the unique local separable model of $T$.
  Every model of $T$ admits a separable submodel, and every separable model of $T$ admits at least one local component, isomorphic to $M_0$.
  Therefore, $T$ is complete and $M_0$ is its prime model.
  In addition, every $1$-type is realised in some local component of a separable model.
  Therefore, all $1$-types are isolated.

  Let us consider a $2$-type $p$.
  We may realise $p$ in a separable model by a pair $(a,b)$.
  If $p$ is isolated, then $p$ is realised in $M_0$, so $a \sim b$.
  Conversely, if $p$ is not isolated, then it cannot be realised in the local component of $a$, so $a \nsim b$.
  We conclude that $a \sim b$ if and only if $\tp(a,b)$ is isolated.

  Consider now any type $p \in \tS_n(T)$, realised by a tuple $a$ in some separable model $M$.
  Given our characterisation of $\sim$, it is immediate that the relation $\sim_p$ of \autoref{dfn:LocallyIsolated} is an equivalence relation.
  Assume now that $q \in \tS_n(T)$ is another type with the same relation $\sim_q$ and the same restrictions $p_I = q_I$ to the equivalence classes.
  If $n = 0$, then $q = p$ since $T$ is complete, so we may assume that $n > 0$.
  We may realise $q$ in a sufficiently saturated model $N \vDash T$, say by $b$.
  In particular, $N$ has at least as many local components as $M$.
  Let $M_\alpha \preceq M$ be a local component of $M$, and let $I = I_\alpha = \{i < n : a_i \in M_\alpha\}$.
  If $I \neq \emptyset$, then it is an equivalence class of $\sim_p$, and there exists an elementary embedding $f_\alpha\colon M_\alpha \rightarrow N$ that sends $a_I \mapsto b_I$.
  If $I = \emptyset$, then we can choose an elementary embedding $f_\alpha\colon M_\alpha \rightarrow N$ into an unused local component of $N$.
  The union $f = \bigcup_\alpha f_\alpha$ sends $a \mapsto b$.
  It is moreover elementary on each local component and sends distinct local components to distinct ones, so it is elementary.
  Therefore, $p = q$, proving that the types of $T$ are indeed locally isolated.

  For the converse implication, it follows from \autoref{lem:LocallyIsolatedStep2} that $(T,{\sim})$ is weakly local.
  It also follows that every local model is atomic, so local $\aleph_0$-categoricity follows from the uniqueness of the separable atomic model.

  The set of isolated types is open by \autoref{lem:LocallyIsolatedStep3} and it is dense because the theory admits a prime model.
The other two items in the second part of the theorem are clear.
\end{proof}

\begin{dfn}
  \label{dfn:LocallyA0Categorical}
  A theory is \emph{locally $\aleph_0$-categorical} if it is locally $\aleph_0$-categorical relative to the only possible locality relation (according to \autoref{thm:LocalRyllNardzewski}), namely,
  \begin{gather*}
    a \sim b
    \quad \Longleftrightarrow \quad
    \tp(a,b) \text{ is isolated}.
  \end{gather*}

  A structure $M$ is \emph{locally $\aleph_0$-categorical} if $\Th(M)$ is locally $\aleph_0$-categorical and $M$ is its unique separable local model.
\end{dfn}

It follows from \autoref{thm:DefinableMetrisation} that every locally $\aleph_0$-categorical theory, or structure, admits a localising metric $d$ in the sense of \autoref{dfn:WeaklyLocalTheory}.
This is a generalised definable metric such that for any $n$-tuple $a$, $\tp(a)$ is isolated if and only if $d(a_i,a_j) < \infty$ for all $i,j < n$.
In particular, it is finite-valued (i.e., a metric) on every locally $\aleph_0$-categorical structure.

If $M$ is an $\aleph_0$-categorical structure, then it is, in particular, locally so.
In addition, every type of $T = \Th(M)$ is realised in $M$ and every model of $T$ is local.
It follows that every definable metric on $M$ is localising and bounded.

\begin{prp}
  \label{prp:LocallyA0CategoricalLocalisingUnbounded}
  Let $T$ be a locally $\aleph_0$-categorical, non-$\aleph_0$-categorical theory and let $M$ be its unique separable local model.
  Let $d$ be a definable metric on $M$.
  Then $d$ localising if and only if it is unbounded.
\end{prp}
\begin{proof}
  Since $T$ is not $\aleph_0$-categorical, it admits a non-isolated type, so a localising metric must be unbounded.
  For the opposite implication, assume that $d$ is unbounded, and let $p(x,y) \in \tS_2(T)$.

  If $p$ is isolated, then it is realised in $M$, so $d(x,y)^p < \infty$.
  Conversely, assume that $p$ is non-isolated, and let $q(x)$ and $r(y)$ denote its restrictions to each variable.
  By compactness, there exists a type $p_0 \in \tS_2(T)$ such that $d(x,y)^{p_0} = \infty$.
  Since $\tS_1(T)$ has finite diameter, there exists $p_1 \in \tS_2(T)$ that extends $q(x)$ and $r(y)$, such that $d(p_1,p_0) < \infty$.
  Then $d(x,y)^{p_1} = \infty$, so $p_1$ is not realised in $M$, and is non-isolated.
  By local isolation of types, $p_1 = p$, so $d(x,y)^p = \infty$.
  Therefore, $d$ is localising.
\end{proof}

We extend a localising metric to $n$-tuples by taking the maximum:
\begin{gather*}
  d(a,b) = \max_{i<n} d(a_i,b_i).
\end{gather*}
We may then equip the type spaces $\tS_n(T)$ with the corresponding generalised metric structure.

\begin{cor}
  \label{cor:LocallyA0CategoricalIsolatedProper}
  Let $T$ be a locally $\aleph_0$-categorical theory.
  Let $\tS_n(T)$ be equipped with the generalised metric structure induced by a localising metric $d$, and let $\fI_n(T) \subseteq \tS_n(T)$ denote the collection of isolated types.

  Then $\bigl( \fI_n(T), d \bigr)$ is a proper metric space for all $n$.
\end{cor}
\begin{proof}
  If $p,q \in \fI_n(T)$, then they can be realised in the prime model of $T$, and therefore at finite distance.
  Therefore $\bigl( \fI_n(T), d \bigr)$ is a metric space.
  Assume now that $K \subseteq \fI_n(T)$ is metrically closed and bounded.
  Let $(p_k)_k$ be any sequence in $K$.
  By passing to a subsequence, we may assume that $p_k \rightarrow p \in \tS_n(T)$ in the logic topology.
  Since $K$ is bounded, $d(p_0,p_k) \leq C$ for some fixed $C$, and since the metric on $\tS_n(T)$ is lower semi-continuous, $d(p_0,p) \leq C < \infty$.
  It follows that $p$ is isolated as well.
  Therefore $d(p_k,p) \rightarrow 0$, and in particular, $p \in K$.
  We have shown that every sequence in $(K,d)$ admits a converging subsequence, so $(K,d)$ is compact.
\end{proof}

\begin{cor}
  \label{cor:LocallyA0CategoricalAutGroup}
  Let $M$ be a locally $\aleph_0$-categorical structure and let $G = \Aut(M)$, $T = \Th(M)$.
  Equip $M$ with a localising metric.
  Then the following hold:
  \begin{enumerate}
  \item
    \label{item:LocallyA0CategoricalAutGroupRealisedTypes}
    For each $n$, the orbit closure space $M^n \dslash G$ can be identified isometrically with $\fI_n(T) \subseteq \tS_n(T)$, the space of isolated $n$-types of $T$, via the map $[a] \mapsto \tp(a)$.
    It is a proper metric space for each $n$, compact for $n = 1$.
  \item
    \label{item:LocallyA0CategoricalAutGrouplRpc}
    The automorphism group $G$ is locally Roelcke precompact and the action $G \curvearrowright M$ is coarsely proper.
  \end{enumerate}
\end{cor}
\begin{proof}
  Since $M$ is the prime model of $T$, the isometric identification $M^n \dslash G \simeq \fI_n(T)$ holds by \autoref{fct:IsolatedTypeDistance}.
  By \autoref{cor:LocallyA0CategoricalIsolatedProper}, $M^n \dslash G$ is a proper metric space.
  For $n = 1$ we have $\fI_1(T) = \tS_1(T)$, so the type metric on $\tS_1(T)$ is compatible with the logic topology, and $M \dslash G$ is compact.
  This proves \autoref{item:LocallyA0CategoricalAutGroupRealisedTypes}, and \autoref{item:LocallyA0CategoricalAutGrouplRpc} follows from \autoref{thm:LocallyRoelckePrecompactProper}.
\end{proof}

The construction of \autoref{dfn:GroupActionStructure} provides us with an important class of examples of locally $\aleph_0$-categorical structures.

\begin{thm}
  \label{thm:GroupActionStructure}
  Let $(A, d)$ be a complete metric space and let $G \curvearrowright A$ by isometries.
  Assume that $A \dslash G$ is compact and $A^n \dslash G$ is a proper metric space for every $n$.
  Then the structure $A^G$ is locally $\aleph_0$-categorical, its theory eliminates quantifiers, and $d$ is a localising metric.
\end{thm}
\begin{proof}
  Let $T = \Th(A^G)$.
  Since $A \dslash G$ is compact, $A \dslash G = \tS_1(T)$.
  Consider the diagonal embedding $A \dslash G \rightarrow A^n \dslash G$ sending $[a] \mapsto [a,a,\ldots,a]$.
  This map is isometric, so its image $\Delta_n \subseteq A^n \dslash G \subseteq \tS_n(T)$ is compact.

  Let us prove that a type $p \in \tS_n(T)$ is isolated if and only if $d(x_i,x_j)^p < \infty$ for all $i,j < n$.
  If $p$ is isolated, then it is realised in the prime model $A^G$, where all distances are finite.
  For the converse, we may choose $\alpha \in \bR$ with $\max_{i<j<n} d(x_i,x_j)^p < \alpha < \infty$.
  The set
  \begin{gather*}
    \Delta_n^\alpha = \bigl\{ q \in A^n \dslash G : d(q,\Delta_n) \leq \alpha\bigr\}
  \end{gather*}
  is closed and bounded in $A^n \dslash G$, and therefore compact by properness.
  Let
  \begin{gather*}
    U^\alpha = \Bigl\{ q \in \tS_n(T) : \max_{i<j<n} d(x_i,x_j)^q < \alpha\Bigr\}.
  \end{gather*}
  This is an open set, $p \in U^\alpha$, and $U^\alpha \cap (A^n \dslash G) \subseteq \Delta_n^\alpha$.
  Since $A^n \dslash G$ is dense in $\tS_n(T)$ and $\Delta_n^\alpha$ is compact, $p \in \Delta_n^\alpha \subseteq A^n \dslash G$, so $p$ is isolated.

  Towards quantifier elimination, let us consider the interpretation of atomic formulas at infinite distance.
  The language of $A^G$ has two kinds of predicate symbols, namely the metric symbol $d$ and the type distance predicates $D_p$.
  If $p \in A^n \dslash G$, $a \in A^n$, and $i,j < n$, then $D_p(a) \geq \frac12  \bigl(d(a_i,a_j) - d(x_i,x_j)^p \bigr)$.
  Therefore, if $b$ is an $n$-tuple in some model of $T$ with $d(b_i,b_j) = \infty$, then $D_p(b) = \infty$.

  Let us consider a tuple $a$ in a model of $T$.
  We may partition $a$ into finite distance components, call them $(a_0,\ldots,a_{k-1})$.
  By our first claim, each $p_i(x_i) = \tp(a_i)$ is isolated, and $\tp^\qf(a)$ tells us that $D_{p_i}(x_i) = 0$, i.e., that $\tp(a_i) = p_i$.
  Therefore, $\tp^\qf(a)$ determines the partition into finite-distance components, as well as the (isolated) type of each component.
  Conversely, our analysis of the atomic formulas tells us that $\tp^\qf(a)$ is determined by the partition of $a$ to finite distance components, together with the quantifier-free type of each component.
  We conclude that the quantifier-free type of $a$ consists exactly of the partition and the quantifier-free types of the components.

  We need to make one last technical observation.
  Let $p \in \tS_n(T)$ (not necessarily isolated), $q = [b] \in A \dslash G = \tS_1(T)$ (automatically isolated), and assume that $A$ has infinite diameter.
  Choose a sequence $(a_k)_k$, $a_k \in A^n$ such that $\tp(a_k) \rightarrow p$.
  Since $A$ is unbounded, we may find $c_k \in A^n$ at distance at least $k$ from $a_k$.
  Let $\delta_k = d\bigl( [c_k], [b] \bigr) = d\bigl( c_k, q \bigr)$.
  Then there exists $b_k \in [b]$ such that $d(b_k,c_k) < \delta_k + 1$.
  Finally, since $A \dslash G$ is compact, it has finite diameter $R$.
  To summarise, we found $a_k$ and $b_k$ in $A$ such that $\tp(a_k) \rightarrow p$, $\tp(b_k) = q$, and $d(b_k, a_k) \geq k - R - 1$.
  Now let $r \in \tS_{n+1}(T)$ be a limit point of the sequence $\tp(a_k,b_k)$, let $M \vDash T$ be sufficiently saturated, and let $a' \in M^n$ realise $p$.
  Then there exists $b' \in M$ such that $\tp(a',b') = r$, and then $\tp(b') = q$, and $b'$ is at infinite distance from $a'$.

  We can now prove that $T$ eliminates quantifiers using the back-and-forth criterion.
  For this, let $M \vDash T$ be sufficiently saturated, and let $a,a' \in M^n$ be such that $a \equiv^\qf a'$.
  Then $a$ and $a'$ are partitioned in the same manner into finite distance components, as $(a_0,\ldots,a_{k-1})$ and as $(a'_0,\ldots,a'_{k-1})$, and $a_i \equiv a_i'$ for each $i < k$.
  Let $b \in M$, and consider two cases.

  Assume first that $b$ is at finite distance from one of the components of $a$, say from $a_0$.
  Then we can find $b' \in M$ such that $a_0b \equiv a'_0b'$.
  The partition of $ab$ into finite distance components is $(a_0b,a_1,\ldots, a_{k-1})$, and similarly for $a'b'$.
  By our characterisation of quantifier-free types, we have $ab \equiv^\qf a'b'$.

  The other case to consider is when $b$ is at infinite distance from $a$.
  Then $A$ has infinite diameter, and we have seen that there exists $b' \in M$ at infinite distance from $a'$ such that $b \equiv b'$.
  In this case, the partition of $ab$ is $(a_0,\ldots,a_{k-1},b)$, similarly for $a'b'$, and again $ab \equiv^\qf a'b'$.

  This completes the proof of the back-and-forth criterion, and therefore of quantifier elimination in $T$.
  Together with our characterisation of quantifier-free types, this implies that all types are locally isolated.
  In addition, $d(a,b) < \infty$ if and only if $\tp(a,b)$ is isolated, so $d$ is a localising metric, and the proof is complete.
\end{proof}

\begin{cor}
  \label{cor:GroupActionStructureLFromGroup}
  Let $G$ be a locally Roelcke precompact Polish group, equipped with a coarsely proper, left-invariant, compatible metric $d_L$.
  Then the structure $\widehat{G}_L^G$ associated to the action $G \curvearrowright \widehat{G}_L$ is locally $\aleph_0$-categorical.

  In particular, a Polish group $G$ is locally Roelcke precompact if and only if it is the automorphism group of some locally $\aleph_0$-categorical structure.
\end{cor}
\begin{proof}
  The space $\widehat{G}_L \dslash G$ is a single point, and for every $n$, $\widehat{G}_L^n \dslash G$ is proper by \autoref{cor:LocallyRoelckePrecompactAction}.
  Now apply \autoref{thm:GroupActionStructure}.
\end{proof}

A structure (in a separable language) is compact if and only if it is the unique model of its theory, up to isomorphism,  if and only if it is $\aleph_0$-categorical and its automorphism group is compact.
The local analogue is contained in the following.

\begin{thm}
  \label{thm:LocallyCompact}
  Let $T$ be a locally $\aleph_0$-categorical theory, let $M$ be its unique local separable model, equipped with a localising metric $d$, and let $G = \Aut(M)$.
  Then the following are equivalent:
  \begin{enumerate}
  \item
    The metric space $(M, d)$ is proper.
  \item
    The group $G$ is locally compact.
  \item Every elementary embedding of $M$ in a local model is an isomorphism.
  \item The structure $M$ is the unique local model of $T$ up to isomorphism.
  \end{enumerate}
  When $T$ is $\aleph_0$-categorical, these conditions hold if and only if $T$ admits a unique model up to isomorphism, which must be compact.
  When $T$ is not $\aleph_0$-categorical, these are further equivalent to
  \begin{enumerate}[resume]
  \item The theory $T$ is $\kappa$-categorical for some (every) uncountable cardinal $\kappa$.
  \end{enumerate}
\end{thm}
\begin{proof}
  \begin{cycprf}
  \item
    The isometry group of a proper metric space is locally compact, as is any closed subgroup thereof.
  \item
    By the homogeneity of $M$, we can identify $\widehat{G}_L$ with the semigroup of elementary embeddings of $M$ into itself.
    In a locally compact group, the left uniformity is complete, so $G = \widehat{G}_L$ and every elementary embedding of $M$ into itself is an isomorphism.

    Stated equivalently, $M$ admits no proper separable local elementary extension.
    By Downward Löwenheim--Skolem, it follows that $M$ admits no proper local elementary extension, which is what we want.
  \item
    Since $M$ embeds elementarily in every model of $T$.
  \item[\impfirst]
    Let us first show that $T$ has the following property:
    \begin{enumerate}
    \item[$(*)$]
      For every $R > r > 0$, there exists $N \in \bN$ such that $T$ forbids the existence of a sequence $(a_i : i < N)$ with $r \leq d(a_i,a_j) \leq R$ for all $i < j$.
    \end{enumerate}
    Indeed, assume that no such $N$ exists for $0 < r < R$.
    Then by compactness, $T$ admits a model containing an uncountable sequence $(a_\alpha : \alpha < \aleph_1)$ such that $r \leq d(b_\alpha,b_\beta) \leq R$ for every $\alpha < \beta < \aleph_1$.
    The sequence must belong to a single local component, which is a non-separable local model of $T$.

    By $(*)$, every local model of $T$ is proper.
  \item[\impnum{4}{5}]
    We assume that $T$ is not $\aleph_0$-categorical, so $M$ cannot be compact, and $T$ admits models of any density character.
    Any model of $T$ of uncountable density character $\kappa$ must consist of $\kappa$ many local components, which are all isomorphic to $M$.
    It follows from \autoref{thm:LocalRyllNardzewski} that any two such models are isomorphic.
  \item[\impfirst]
    As above, it is enough to prove property $(*)$.
    Let $\kappa$ be uncountable.
    We assume that $T$ is not $\aleph_0$-categorical, so $M$ cannot be bounded, and $T$ admits models with arbitrarily many local components.
    Using Downward Löwenheim--Skolem, $T$ admits a model consisting of $\kappa$ many separable local components.
    If $(*)$ fails, then by a similar argument, $T$ admits a local model of density character $\kappa$.
    These two models cannot be isomorphic, so $T$ cannot be $\kappa$-categorical.
  \end{cycprf}
\end{proof}

\begin{rmk}
  While a compact metric structure is automatically $\aleph_0$-categorical, a proper metric structure is not necessarily locally $\aleph_0$-categorical---see \autoref{exm:ZOrd}.
\end{rmk}


\section{Definability in locally $\aleph_0$-categorical structures}
\label{sec:Definability}

Our first goal for this section is to describe the full space of types $\tS_n(T)$ for a locally $\aleph_0$-categorical theory $T$ in terms of the isolated types.
This is, in essence, equivalent to describing the algebra of definable predicates on a locally $\aleph_0$-categorical structure $M$.
We start with some general considerations.

Recall that when $G$ is a topological group, coarsely bounded subsets, in the sense of Rosendal \cite{Rosendal:CoarseGeometry}, form an ideal of subsets of $G$ which contains all compact sets.
When $G$ is Polish and locally bounded, it admits a compatible left-invariant metric $d_L$ which is coarsely proper, namely, such that the $d_L$-bounded sets are precisely the coarsely bounded sets.
When $G$ is Polish and locally Roelcke precompact (which implies locally bounded), the coarsely bounded sets are exactly the Roelcke precompact ones.

For the following definition we do not need a topology but just a coarse structure, i.e., a reasonable notion of when a sequence $(g^k)_k$ of elements of $G$ ``diverges to infinity''.
This can be coded, for example, by a left-invariant pseudo-metric $d_L$ on $G$, where we say that $g^k \rightarrow \infty$ if $d_L(1,g^k) \rightarrow \infty$, without requiring $d_L$ to arise from a topological structure.
In this section, we will often consider sequences in Cartesian powers of $G$ and will need double indices.
In particular, superscripts on elements will denote indices and not powers in the group.

\begin{dfn}
  \label{dfn:ProperlyInvariant}
  Let $G$ be a group equipped with a left-invariant pseudo-metric $d_L$, acting on a set $X$.
  Fix $n \geq 1$ and let $x$ denote a tuple in $X^n$.
  Let $\sim$ be an equivalence relation on $n$, and let $n/{\sim}$ be the associated partition of $n$ into equivalence classes.
  We will say that $\sim$ is \emph{trivial} if it has only one equivalence class.
  \begin{itemize}
  \item For a sequence $(g^k)_k$ of elements of $G^{n/{\sim}}$, write that $g^k \rightsquigarrow \infty$ if $d_L(g^k_I, g^k_J) \rightarrow_k \infty$ for all $I \neq J$.
    In particular, if $\sim$ is trivial, then $g^k \rightsquigarrow \infty$ vacuously.
  \item For $I \in n/{\sim}$, we denote by $x_I \in X^I$ the restriction of $x$ to the index set $I$.
  \item For $g = (g_I : I \in n/{\sim}) \in G^{n/{\sim}}$, we define $gx \in X^n$ by letting $(g x)_I = g_I x_I$ for all $I \in n/{\sim}$, giving rise to an action $G^{n/{\sim}} \curvearrowright X^n$.
  \item
    Let $K$ be a topological space and $P\colon X^n \rightarrow K$ be a map.
    We say that $P$ is \emph{$\sim$-properly $G$-invariant} if for every sequence $(g^k)_k$ of elements of $G^{n / {\sim}}$ and every $x \in X^n$, if $g^k \rightsquigarrow \infty$, then the sequence $P(g^k x)$ converges.
    If this holds for every equivalence relation $\sim$ on $n$, we say that $P$ is \emph{properly $G$-invariant}.
  \end{itemize}
\end{dfn}

When $\sim$ is trivial, we always have $g^k \rightsquigarrow \infty$ in $G$, so $P$ is $\sim$-properly $G$-invariant if and only if it is $G$-invariant, whence the terminology.
In particular, if $P$ is properly $G$-invariant, it is $G$-invariant.

\begin{rmk}
  \label{rmk:ProperlyInvariantDichotomy}
  We have the following dichotomy:
  \begin{itemize}
  \item If $d_L$ is bounded, then we can never have $g^k \rightsquigarrow \infty$ for a non-trivial $\sim$.
    Therefore, a predicate $P$ is properly $G$-invariant if and only if it is $G$-invariant.
  \item If $d_L$ is unbounded, then for every $\sim$, we can find a sequence $g^k \rightsquigarrow \infty$ in $G^{n / {\sim}}$.
    If $P$ is properly $G$-invariant, then the limit $P(g^k x)$ depends only on $x$ and $\sim$, and will be denoted by $P_\sim(x)$.
  \end{itemize}
\end{rmk}

\begin{dfn}
  \label{dfn:CoarselyCompatible}
  Let $X$ be a metric space, and $G \curvearrowright X$ by isometries.
  Let $d_L$ be a left-invariant pseudo-metric on $G$, and assume that $d_L(g^k,h^k) \rightarrow \infty$ if and only if $d(g^k x, h^k x) \rightarrow \infty$ for some (equivalently, any) $x \in X$.
  Then we say that $d_L$ is \emph{coarsely compatible} with the action $G \curvearrowright X$.

  We say that a predicate $P\colon X^n \rightarrow K$ is \emph{($\sim$-)properly $G$-invariant} if it is so relative to any coarsely compatible, left-invariant pseudo-metric.
\end{dfn}

Clearly, a coarsely compatible, left-invariant pseudo-metric always exists (e.g., $d_L(g,h) = d(gx,hx)$, where $x \in X$).
Moreover, all such pseudo-metrics induce the same notion of $g^k \rightsquigarrow \infty$, and therefore the same notion of a properly $G$-invariant predicate.

When $M$ is a locally $\aleph_0$-categorical structure equipped with a localising metric and $G = \Aut(M)$, then the action $G \curvearrowright M$ is coarsely proper.
Therefore, a left-invariant, compatible pseudo-metric $d_L$ is coarsely proper on $G$ if and only if it is coarsely compatible with the action $G \curvearrowright M$.

\begin{thm}
  \label{thm:LocallyA0CategoricalDefinable}
  Let $M$ be a locally $\aleph_0$-categorical structure equipped with a localising metric and $G = \Aut(M)$.
  Let $K$ be a compact space (for example, a closed bounded subset of $\bR$).

  Then a predicate $P \colon M^n \to K$ is definable (without parameters) if and only if it is uniformly continuous and properly $G$-invariant.
\end{thm}
\begin{proof}
  We may equip $M$ with a localising metric, so that the action $G \curvearrowright M$ is coarsely proper by \autoref{cor:LocallyA0CategoricalAutGroup}.

  Assume first that $P$ is definable.
  It is a standard fact that every definable predicate is uniformly continuous.
  In order to check that it is properly $G$-invariant, assume that $a \in M^n$, $\sim$ is an equivalence relation on $n$, and $g^k \rightsquigarrow \infty$ in $G^{n / {\sim}}$ as in \autoref{dfn:ProperlyInvariant}.
  Let $T = \Th(M)$.
  Let $p^k = \tp(g^k a) \in \tS_n(T)$, and for $I \in n / {\sim}$, let $p_I = \tp(a_I) \in \tS_I(T)$.

  Let $q$ be a limit point of $(p^k)_k$.
  If $i \sim j$, then $d(x_i,x_j)^q = d(a_i,a_j)$.
  On the other hand, if $i \nsim j$, then $d(x_i,x_j)^q = \lim_k d( g^k_I a_i, g^k_J a_j) = \infty$, where $I$ and $J$ are the respective classes of $i$ and $j$.
  By local isolation of types, $q$ is the unique type such that ${\sim}_q = {\sim}$ and $q_I = p_I$ for each $I \in n / {\sim}$.
  Therefore, $p^k \rightarrow q$, and $P(g^k a) = P(p^k) \rightarrow P(q)$.

  Conversely, assume that $P$ is uniformly continuous and properly $G$-invariant.
  By \autoref{cor:LocallyA0CategoricalAutGroup}, we may identify $M^n \dslash G$, as a metric space, with $\fI_n(T) \subseteq \tS_n(T)$, the space of isolated $n$-types of $T$.
  It follows from continuity and $G$-invariance that $P$ induces a continuous function $M^n \dslash G \to \bR$.
  Therefore, all that is left to show is that $P$ admits a continuous extension (necessarily unique) to $\tS_n(T)$.
  It will suffice to show that if $p^k \in \fI_n(T)$ and $p^k \rightarrow q \in \tS_n(T)$, then $P(p^k)$ converges.

  We may assume that $p^k = \tp(b^k)$ and $q = \tp(c)$, where $b^k \in M^n$ but $c$ may belong to another model of $T$.
  Recall that $\sim_q$ is the equivalence relation on $n$ where $i \sim_q j$ if $d(c_i,c_j) < \infty$.
  If $I \in n / {\sim_q}$, then $q_I = \tp(c_I)$ is isolated, so it is realised by a tuple $a_I \in M^I$.
  Since the classes $I$ are disjoint, we may view the $a_I$ as subtuples of some $a \in M^n$.

  If $I \in n / {\sim_q}$, then $\tp(b^k_I) \rightarrow \tp(c_I) = \tp(a_I)$ in $\fI_I(T)$.
  Since, on isolated types, the type metric is compatible with the logic topology, there exist $g^k_I \in G$ such that $d(b^k_I, g^k_I a_I) \rightarrow 0$.
  Equivalently, there exist $g^k \in G^{n / {\sim_q}}$ such that $d(b^k, g^k a) \rightarrow 0$.
  If $i \nsim_q j$, then $d(b^k_i,b^k_j) \rightarrow d(c_i,c_j) = \infty$, and therefore $d(g^k_I a_i,g^k_J a_j) \rightarrow \infty$, where $I$ and $J$ are the respective classes.
  Therefore $g^k \rightsquigarrow \infty$, so $P(g^k a) \rightarrow P_\sim(a)$ by proper $G$-invariance, and since $P$ is assumed uniformly continuous, $P(b^k) \rightarrow P_\sim(a)$ as well.
\end{proof}

Recall that two structures on the same set (but possibly in different languages) are \emph{bi-definable} if they have the same definable predicates.

\begin{cor}
  \label{cor:LocallyA0CategoricalDefinitionEquivalent}
  Let $M_0$ and $M_1$ be two locally $\aleph_0$-categorical structures in possibly distinct languages, with a common underlying set $M$, common automorphism group $G$, and uniformly equivalent definable metrics.
  Then $M_0$ and $M_1$ are bi-definable.

  In particular, if $M$ is a locally $\aleph_0$-categorical structure and $G = \Aut(M)$, then $M$ is bi-definable with $M^G$.
\end{cor}
\begin{proof}
  Since the metrics on the two structures are uniformly equivalent, the automorphism group $G$ carries the same locally Roelcke precompact topology for both structures, and therefore the same coarse structure.
  We may assume that the definable metrics are isolating, in which case $G \curvearrowright M_i$ is coarsely proper, implying that the localising metrics are coarsely equivalent as well, and give rise to the same notion of proper $G$-invariance.
  We conclude using the characterisation of definable predicates in \autoref{thm:LocallyA0CategoricalDefinable}.

  In particular, if $M$ is locally $\aleph_0$-categorical, then $M^G$ is locally $\aleph_0$-categorical by \autoref{thm:GroupActionStructure}, so the main clause applies.
\end{proof}

We have given a description of the algebra of definable predicates on $M^n$, from which the space of types $\tS_n(T)$ can be obtained as the Gelfand dual.
Next we provide a more explicit construction of $\tS_n(T)$ as the completion of $M^n \dslash G = \fI_n(T)$ with respect to an appropriate uniform structure.
This uniformity can be defined in the abstract setting of a group acting isometrically on a metric space.

\begin{dfn}
  \label{dfn:InvariantInfinityUniformity}
  Let $G$ be a group acting on a metric space $(X, d)$ by isometries and let $n \in \bN$.
  \begin{itemize}
  \item We recall that for $x \in X^n$ and $I \subseteq n$, $x_I \in X^I$ denotes the restriction of the tuple $x$ to $I$.
    Notice that $[x_I] \in X^I \dslash G$ only depends on $[x] \in X^n \dslash G$ and on $I$, allowing us to define $[x]_I = [x_I]$.
    For $p,q \in X^n \dslash G$, let $d_I(p,q) = d(p_I,q_I)$, so $d_I \leq d_n = d$ on $X^n \dslash G$.
  \item For $p = [x] \in X^n \dslash G$, $i < n$, and $m \in \bN$, define
    \begin{gather*}
      I(x,i,m) = \bigl\{j < n : d(x_i,x_j) < 2^m \bigr\}.
    \end{gather*}
    This only depends on $p = [x]$, $i$, and $m$, so we may also denote it by $I(p,i,m)$.
  \item For $m \in \bN$, define
    \begin{multline*}
      U^\infty_m = \Bigl\{ (p,q) \in (X^n \dslash G)^2 : d_{I(p,i,m)}(p,q) < 2^{-m} \ \\
      \text{and} \ d_{I(q,i,m)}(p,q) < 2^{-m}\ \text{for all} \ i<n \Bigr\}.
    \end{multline*}
  \item We denote by $\cU^\infty$ the collection of subsets of $(X^n \dslash G)^2$ that contain $U^\infty_m$ for some $m$.
  \end{itemize}
\end{dfn}

\begin{lem}
  \label{lem:InvariantInfinityUniformity}
  Let $G$ be a group acting by isometries on a metric space $(X, d)$ and let $n \in \bN$.
Then:
  \begin{enumerate}
  \item For every $m$, the set $U^\infty_m$ is symmetric.
  \item The sequence $(U^\infty_m)_m$ is decreasing, and the intersection $\bigcap_m U^\infty_m$ is the diagonal of $(X^n \dslash G)^2$.
  \item For every $m$, we have $U^\infty_{m+1} \circ U^\infty_{m+1} \subseteq U^\infty_m$.
  \end{enumerate}
  Consequently, the family $(U^\infty_m)_m$ generates a metrisable uniform structure on $X^n \dslash G$, which is exactly $\cU^\infty$.
  It is coarser than the uniform structure induced by the metric on $X^n \dslash G$, and induces the same topology as the metric.
\end{lem}
\begin{proof}
  The first two assertions are immediate.
  For the third, let $p = [x]$, $q = [y]$, $r = [z]$, and assume that $(p,q),(q,r) \in U^\infty_{m+1}$.
  Let $i,j < n$, and assume that $j \in I(p,i,m)$, i.e., that $d(x_i,x_j) < 2^m$.
  Then there exists $g \in G$ such that $d\bigl( x_i x_j, g(y_i y_j)\bigr) < 2^{-m-1}$, so $d(y_i, y_j) < 2^m + 1 \leq 2^{m+1}$.
  It follows that
  \begin{gather*}
    I(p,i,m) \subseteq I(p,i,m+1) \cap I(q,i,m+1).
  \end{gather*}
  Therefore,
  \begin{equation*}
    \begin{split}
      d_{I(p,i,m)}(p,r)
      &\leq d_{I(p,i,m)}(p,q) + d_{I(p,i,m)}(q,r) \\
      &\leq d_{I(p,i,m+1)}(p,q) + d_{I(q,i,m+1)}(q,r)
      < 2^{-m}.
    \end{split}
  \end{equation*}
  One shows similarly that $d_{I(r,i,m)}(p,r) < 2^{-m}$, so $(p,r) \in U^\infty_m$.
  Therefore, $U^\infty_{m+1} \circ U^\infty_{m+1} \subseteq U^\infty_m$.

  For the last assertion, note that $\bigl\{ (p,q) : d(p,q) < 2^{-m} \bigr\} \subseteq U^\infty_m$, so $\cU^\infty$ is coarser than the uniform structure defined by the metric on $X^n \dslash G$.
  Conversely, for a fixed $p = [x] \in X^n \dslash G$, if $\max_{i<j} d(x_i,x_j) < 2^m$ and $(p,q) \in U^\infty_m$, then $d(p,q) < 2^{-m}$.
  Therefore $\cU^\infty$ defines the same topology as $d$ on $X^n \dslash G$.
\end{proof}

We defined $\cU^\infty$ on $X^n \dslash G$, which we may pull back to a uniform structure on $X^n$.
A predicate $P\colon X^n \rightarrow \bR$ is $\cU^\infty$-uniformly continuous in this sense if and only if it factors through a function $\overline{P}\colon X^n \dslash G \rightarrow \bR$, and $\overline{P}$ is $\cU^\infty$-uniformly continuos.

\begin{lem}
  \label{lem:Uinfty-implies-propGinv}
  Let $(X, d)$ be a metric space and $G \actson X$ an action by isometries.
  Then every $\cU^\infty$-uniformly continuous bounded function $P \colon X^n \to \bR$ is $d$-uniformly continuous and properly $G$-invariant.
\end{lem}
\begin{proof}
  By \autoref{lem:InvariantInfinityUniformity}, $P$ is $d$-uniformly continuous.
  To see that it is properly $G$-invariant, fix an equivalence relation $\sim$ on $n$.
  Let $(g^k)_k$ be a sequence of elements of $G^{n / {\sim}}$ such that $g^k \rightsquigarrow \infty$, and let $x \in X^n$.
  Let $p^k = [g^k x]$.
  For every $m$, and every $k$ large enough relative to $m$, the set $I(p^k,i,m)$ is contained in the $\sim$-class of $i$.
  It follows that the sequence $(p^k)_k$ is $\cU^\infty$-Cauchy, so by assumption, $P(g^kx)$ converges.
\end{proof}

We can now show that for a locally $\aleph_0$-categorical structure $M$, the space of all $n$-types of $\Th(M)$ can be constructed explicitly as the $\cU^\infty$-completion of $M^n \dslash \Aut(M)$.
We shall use the following easy and standard fact: if the algebra of uniformly continuous, bounded functions on a metrisable uniform space $X$ is separable, then $X$ is precompact.

\begin{prp}
  \label{prp:LocallyA0CategoricalUniformity}
  Let $M$ be a locally $\aleph_0$-categorical structure equipped with a localising metric, and let $G = \Aut(M)$, $T = \Th(M)$.
  Then the identification of $\fI_n(T)$ with $M^n \dslash G$ extends to an identification of $\tS_n(T)$ with the $\cU^\infty$-completion of $M^n \dslash G$.
\end{prp}
\begin{proof}
  Let $\UCB(M^n \dslash G,\cU^\infty)$ denote the algebra of bounded, $\cU^\infty$-uniformly continuous functions on $M^n \dslash G$.
  By \autoref{thm:LocallyA0CategoricalDefinable} and \autoref{lem:Uinfty-implies-propGinv}, we can identify $\UCB(M^n \dslash G,\cU^\infty)$ with a subalgebra of the algebra of continuous functions $C \bigl(\tS_n(T) \bigr)$.
  As the language of $T$ is separable, so is the algebra $C \bigl(\tS_n(T) \bigr)$, and therefore also $\UCB(M^n \dslash G,\cU^\infty)$.
  It follows that the $\cU^\infty$-completion $\widehat{M^n \dslash G}$ is compact; from Gelfand duality, we obtain a continuous, surjective map $\theta \colon \tS_n(T) \to \widehat{M^n \dslash G}$.

  It remains to show that $\theta$ is injective.
  Let $p, q \in \tS_n(T)$ be such that $\theta(p) = \theta(q)$.
  Let $(p^k)_k$ and $(q^k)_k$ be sequences in $M^n \dslash G$ such that $p^k \to p$ and $q^k \to q$ in $\tS_n(T)$.
  Say that $i \sim_p j$ if $\lim_k d(x_i,x_j)^{p_k} = d(x_i,x_j)^p < \infty$, and similarly for $q$.
  Since $\theta(p) = \theta(q)$, the intertwined sequence $(p^k, q^k)_k$ is $\cU^\infty$-Cauchy, so $\sim_p$ and $\sim_q$ coincide.
  Similarly, for each $I \in n / {\sim_p}$, the intertwined sequence $(p^k_I, q^k_I)_k$ is $d$-Cauchy, implying that $p_I = q_I$.
  Using the local isolation of types, we conclude that $p = q$.
\end{proof}


\section{Interpretation in locally $\aleph_0$-categorical structures}
\label{sec:Interpretation}

We turn to questions of interpretability of structures.
For the formal definition of an interpretational expansion and bi-interpretability, we refer the reader to \cite{BenYaacov:LelekGroupoid}, but we briefly recall it here.
Roughly speaking, an \emph{interpreted sort} is a definable subset of a pseudo-metric quotient of a countable Cartesian product of \emph{basic sorts}, where the basic sorts always include the \emph{constant sort} $\{0,1\}$.
An \emph{interpretational expansion} is an expansion that adds an interpreted sort, together with its induced structure, as a new basic sort.
In order to establish that the addition of a new sort is an interpretational expansion, we need to show that, in the expanded structure, the new sort admits a definable bijection with an interpretable sort of the original structure, and that all the newly introduced structure is definable, via this bijection, from the structure on the original sorts.
Two structures are \emph{bi-interpretable} if they admit a common interpretational expansion.

\begin{rmk}
  \label{rmk:InterpretationEquivalenceCompleteTheory}
  A sort can be interpreted in a single structure or in a theory, i.e., uniformly across all its models.
  However, any sort interpretable in a structure is also interpretable in its theory, so for complete theories the two notions are essentially the same.
  To be more precise, if $T = \Th(M)$, then $T$ and $T'$ are bi-interpretable if and only if $M$ is bi-interpretable with some $M'$ such that $T' = \Th(M')$.

  If $M'$ is an interpretational expansion of $M$, then $\Aut(M) \cong \Aut(M')$.
  Therefore, if $M'$ is bi-interpretable with $M$, then $\Aut(M) \cong \Aut(M')$.
  For a similar reason, $M$ is a prime model of its theory if and only if $M'$ is.
\end{rmk}

\begin{dfn}
  \label{dfn:LocalSort}
  Let $T$ be a weakly local theory and let $D$ be an interpretable sort in $T$.
  We say that $D$ is a \emph{local} sort if for every model $M \vDash T$, if $M = \bigsqcup_\alpha M_\alpha$ is its decomposition into local components, then $D^M = \bigsqcup_\alpha D^{M_\alpha}$.
\end{dfn}

\begin{lem}
  \label{lem:LocallyA0CategoricalLocalSort}
  Let $M$ be a locally $\aleph_0$-categorical structure, and let $G = \Aut(M)$.
  Let $D$ be a local interpreted sort of $T = \Th(M)$.
  Then $D^M$, together with its induced structure, is locally $\aleph_0$-categorical as well.
\end{lem}
\begin{proof}
  Let $T_D = \Th(D^M)$ in the induced structure.
  The type space $\tS_n(T_D)$ coincides with $\tS_{D^n}(T)$, the space of types of $n$-tuples in $D$, viewed as an interpreted sort in $T$.
  In particular, type isolation in one sense coincides with the other.

  By the definition of a local sort, every $p \in \tS_1(T_D) = \tS_D(T)$ is realised in a local model of $T$, so it must be isolated.
  Similarly, since each member of $D$ is represented in a unique local (and therefore, atomic) component, the joint isolation relation on $D$ is an equivalence relation.

  Let $N \vDash T$ be an $\aleph_1$-saturated model of $T$ and let $p \in \tS_n(T_D)$.
  Then $p$ is realised by some $a \in (D^N)^n$.
  Let $N = \bigsqcup_\alpha N_\alpha$ be the decomposition of $N$ into local components, and for $i < n$, let $\alpha_i$ be the index of the component such that $a_i \in D^{N_{\alpha_i}}$.

  For each $i$, $\tp(a_i)$ is isolated (since $N_{\alpha_i}$ is atomic).
  Similarly, if $\alpha_i = \alpha_j$, then $\tp(a_i,a_j)$ is isolated.
  Conversely, assume that $\tp(a_i,a_j)$ is isolated.
  Then it is realised in $D^{N'}$ for every $N' \vDash T$, and in particular, in the separable local model, say by $(b,b')$.
  Since $N$ is $\aleph_1$-saturated, there exists an elementary embedding $N' \preceq N$ that identifies $(b,b')$ with $(a_i,a_j)$.
  Since $N'$ is local, it embeds in a single local component of $N$, so $\alpha_i = \alpha_j$.

  We conclude that $\alpha_i = \alpha_j$ if and only if $\tp(a_i,a_j)$ is isolated, which we denote by $i \sim_p j$.
  For each class $I$, $a_I \in (D^{N_{\alpha_I}})^I$ for a unique $\alpha_I$.
  If $I \neq J$ are distinct classes, then $\alpha_I \neq \alpha_J$.
  Therefore, $\tp(a)$ is entirely determined by the types $\tp(a_I)$ for $I \in n / {\sim_p}$.
  This proves local isolation of types, so $T_D$ is locally $\aleph_0$-categorical.

  Finally, $D^M$ is separable and atomic, so it is the prime model of $T_D$, and therefore its unique separable local model.
\end{proof}

If $M$ is an unbounded locally $\aleph_0$-categorical structure, a simple example of a non-local interpreted sort is $M^2$.
In this case, the structure $M^2$ is \emph{not} locally $\aleph_0$-categorical.

Let $T$ be a theory, and let $D_0,D_1$ be two interpreted sorts in $T$.
Equip them with definable metrics $d_0$ and $d_1$, both bounded by one.
Let $(D,d)$ be the quotient of $\{0,1\} \times D_0 \times D_1$ by the following definable pseudo-metric:
\begin{gather*}
  d\bigl( (i,a_0,a_1) , (j,b_0,b_1) \bigr) =
  \begin{cases}
    d_i(a_i,b_i) & i = j, \\
    1 & i \neq j.
  \end{cases}
\end{gather*}
Then $D$ can be identified with the disjoint union $D_0 \sqcup D_1$, where each $D_i$ is equipped with the respective metric $d_i$, and the two sets are at distance $1$ from one another.
The projection from $\{0,1\} \times D_0 \times D_1$ to the first coordinate induces a definable predicate on $D$, equal to $i$ on $D_i$, so each $D_i$ is a definable subset of $D$.
The space of types in $(D_0 \sqcup D_1)^n$ is the topological disjoint union of all spaces of types in products of copies of $D_0$ and $D_1$.

\begin{lem}
  \label{lem:Coquand1}
  Let $M_0$ and $M_1$ be locally $\aleph_0$-categorical structures, with automorphism groups $G_i = \Aut(M_i)$.
  Assume that:
  \begin{itemize}
  \item for each $i$, $D_i$ is a local interpretable sort in $M_i$ (on which $G_i$ acts as well);
  \item there exist uniform homeomorphisms $\theta_M\colon M_0 \simeq D_1$ and $\theta_D\colon D_0 \simeq M_1$; and
  \item there exists a group isomorphism $\varphi\colon G_0 \simeq G_1$ which is compatible with $\theta_M$ and $\theta_D$, by which we mean that for $a \in M_0$, $b \in D_0$, and $g \in G_0$:
    \begin{gather*}
      \varphi(g) \cdot \theta_M(a) = \theta_M(ga),
      \qquad
      \varphi(g) \cdot \theta_D(b) = \theta_D(gb).
    \end{gather*}
  \end{itemize}
  Then $M_0$ and $M_1$ are bi-interpretable.
\end{lem}
\begin{proof}
  Construct the disjoint unions $M_i \sqcup D_i$ for $i = 0,1$ as above.
  Then each is a local interpretable sort in the respective $M_i$, and therefore, by \autoref{lem:LocallyA0CategoricalLocalSort}, a locally $\aleph_0$-categorical structure in its own right.
  Moreover, $M_i \sqcup D_i$ is an interpretational expansion of $M_i$.
  Identifying the underlying sets of $M_0 \sqcup D_0$ and $D_1 \sqcup M_0$ via $\theta_M \sqcup \theta_D$, by \autoref{cor:LocallyA0CategoricalDefinitionEquivalent}, these two structures are bi-definable.
  Thus we have constructed a common interpretational expansion of $M_0$ and of $M_1$, whence the conclusion.
\end{proof}

\begin{rmk}
  \label{rmk:CompactSetInterpretation}
  Let $K$ be a compact metric space.
  As $K$ can be represented as the continuous image of $\set{0, 1}^\bN$, it is interpretable in any structure using only the constant sort $\{0,1\}$.
  Since automorphisms act as the identity on $\{0,1\}$, they also act as the identity on $K$.
  If $T$ is a local theory, $M \vDash T$ and $M = \bigsqcup_\alpha M_\alpha$ is the decomposition of $M$ into local components, then $M \times K$ decomposes as $\bigsqcup_\alpha (M_\alpha \times K)$.
  Therefore, if $M$ is locally $\aleph_0$-categorical, then $M \times K$ is a local sort in the sense of \autoref{dfn:LocalSort}, and therefore, with the induced structure, is locally $\aleph_0$-categorical as well.
\end{rmk}

\begin{lem}
  \label{lem:Coquand2}
  Let $M$ be locally $\aleph_0$-categorical, and let $G = \Aut(M)$.
  Equip $G$ with a compatible, coarsely proper, left-invariant metric $d_L$, and let $\widehat{G}_L$ denote the completion.
  Let the structure $\widehat{G}_L^G$ be as in \autoref{dfn:GroupActionStructure}, relative to the left translation action $G \curvearrowright \widehat{G}_L$.

  Then $\widehat{G}_L^G$ is locally $\aleph_0$-categorical, and it is bi-interpretable with $M$.
\end{lem}
\begin{proof}
  We have a natural identification $G = \Aut(M) = \Aut(\widehat{G}_L^G)$.
  We may assume that $M$ is equipped with a localising metric.

  We may view $M^\bN$ as an interpreted sort in $M$ and equip it with a compatible, definable metric, for example,
  \begin{gather*}
    d(a,b) = \sup_{i \in \bN} \, \bigl( 2^{-i} \wedge d(a_i,b_i) \bigr).
  \end{gather*}
  Fix $\xi \in M^\bN$ that enumerates a dense subset of $M$ and let $p = \tp(\xi)$.
  Since $M$ is prime, $p$ is isolated, so $D = [\xi]$ a definable set.
  The map $G \rightarrow G \xi$, $g \mapsto g\xi$ is a uniform homeomorphism.
  Therefore it extends to a uniform homeomorphism $\widehat{G}_L \rightarrow D$ between the completions.

  Next we interpret $M$ in the structure $\widehat{G}_L^G$ (which is locally $\aleph_0$-categorical by \autoref{cor:GroupActionStructureLFromGroup}).
  By \autoref{lem:CoCompact}, there exists a compact set in $M$ that meets every orbit closure in $M \dslash G$.
  Let us write this set as $\{a_t : t \in K\}$, where $K$ is a compact space and $t \mapsto a_t$ is a homeomorphic embedding into $M$.
  Define a pseudo-metric $\tilde{d}$ on $G \times K$ by
  \begin{gather*}
    \tilde{d}\bigl( (g,t), (h,s) \bigr) = d^M(g a_t, h a_s).
  \end{gather*}
  It is uniformly continuous and therefore extends to a uniformly continuous pseudo-metric on $\widehat{G}_L \times K$.
  As the action $G \curvearrowright (\widehat{G}_L \times K)$ is trivial on the second coordinate, $\tilde{d}$ is $G$-invariant.
  By the coarse properness of the action, if $f_k \rightarrow \infty$ in $G$, then $\tilde{d}\bigl( (g,t), (f_kh,s) \bigr) \rightarrow \infty$, so $\tilde{d}$ is properly $G$-invariant.
  By \autoref{rmk:CompactSetInterpretation}, the product $\widehat{G}_L \times K$, equipped with the induced structure from $\widehat{G}_L^G$, is locally $\aleph_0$-categorical.
  Therefore, by \autoref{thm:LocallyA0CategoricalDefinable}, $\tilde{d}$ is definable (in $\widehat{G}_L^G \times K$, and therefore in $\widehat{G}_L^G$).

  Let $E$ denote the \emph{complete quotient} of $(\widehat{G}_L \times K, \tilde{d})$, i.e., the completion of the quotient by the zero-distance equivalence relation.
  The map $G \times K \rightarrow M$ sending $(g,t) \mapsto ga_t$ has dense image (by the choice of $K$), and induces an isometric bijection $E \simeq M$.

  We have constructed the interpreted sets $D$ in $M$ and $E$ in $\widehat{G}_L^G$, as well as uniform isomorphisms $\theta_D\colon \widehat{G}_L \simeq D$ and $\theta_E\colon E \simeq M$.
  Consider $g,h \in G$ and $t \in K$.
  Then
  \begin{gather*}
    g \theta_D(h) = gh \xi = \theta_D(gh),
    \qquad
    g \theta_E\bigl( [h,t] \bigr) = gha_t = \theta_E\bigl( g [h,t] \bigr).
  \end{gather*}
  By density, the same holds everywhere on $\widehat{G}_L$ or $E$.
  We conclude using \autoref{lem:Coquand1}.
\end{proof}

From this, we deduce an analogue of Coquand's theorem for locally $\aleph_0$-categorical structures.
\begin{thm}
  \label{thm:LocalCoquand}
  Let $M_0$ and $M_1$ be locally $\aleph_0$-categorical structures.
  Then $M_0$ is bi-interpretable with $M_1$ if and only if $\Aut(M_0) \cong \Aut(M_1)$ as topological groups.
\end{thm}
\begin{proof}
  If the structures are bi-interpretable, then, up to an isomorphism, they admit a common interpretational expansion $(M_0,M_1)$, in which case $\Aut(M_0) \cong \Aut(M_0,M_1) \cong \Aut(M_1)$.
  For the converse, let us identify $G = \Aut(M_0) = \Aut(M_1)$.
  Each $M_i$ is bi-interpretable with $\widehat{G}_L^G$, so all three structures admit a common interpretational expansion (in the three sorts $(M_0,M_1,\widehat{G}_L)$), so $M_0$ and $M_1$ are bi-interpretable.
\end{proof}

\begin{cor}
  \label{c:interpr-implies-coarse-eq}
  Suppose that $M_0$ and $M_1$ are locally $\aleph_0$-categorical structures that are bi-interpretable.
  Let $d_i$ be a localising metric on $M_i$ for $i = 0,1$.
  Then the metric spaces $(M_0, d_0)$ and $(M_1, d_1)$ are coarsely equivalent.
\end{cor}
\begin{proof}
  Let $G = \Aut(M_0) \cong \Aut(M_1)$.
  By \autoref{cor:LocallyA0CategoricalAutGroup}, the actions $G \actson M_i$ are coarsely proper.
  They are also cobounded (since $M_i \dslash G$ is compact), and therefore each of the metric spaces $(M_i, d_i)$ is coarsely equivalent to $G$.
\end{proof}


\section{Examples}
\label{sec:examples}

We start with a natural family of examples that have already appeared.
Assume that $G$ is a Polish, locally Roelcke precompact group.
Let $d_L$ be a compatible, coarsely proper, left-invariant metric, and for $g \in G$, let $R_g\colon x \mapsto xg$ denote the associated right translation map.
By \autoref{rmk:GroupActionStructureAlternateLanguage}, the structure $(\widehat{G}_L,d_L,R_g)_{g \in G}$ is bi-definable with $\widehat{G}_L^G$ of \autoref{dfn:GroupActionStructure}, which, by \autoref{cor:GroupActionStructureLFromGroup}, is locally $\aleph_0$-categorical.
In addition:
\begin{itemize}
\item If $S \subseteq G$ generates a dense subgroup, we lose nothing when considering $\widehat{G}_L$ in the sublanguage $(d_L,R_g)_{g \in S}$.
\item When $G$ is locally compact, $d_L$ is complete, so $\widehat{G}_L = G$, and we obtain the structure $(G,d_L,R_g)_{g \in S}$.
\item When $G$ is a discrete, countable group, we may drop $d_L$ in favour of the equality relation, since $d(g_n,h_n) \rightarrow \infty$ if and only if $h_n^{-1}g_n$ leaves every finite set, and this is definable using equality and the right-translation maps.

\item Therefore, when $G$ is finitely generated and $S$ is a generating set, we obtain a locally $\aleph_0$-categorical structure $(G,R_g)_{g \in S}$, which is the labelled right Cayley graph of $G$.
  In that case, the localising metric is the word metric on $G$.
\end{itemize}

\begin{exm}
  \label{exm:ZSucc}
  When $G = \bZ$ and $S = \{1\}$, we obtain the locally $\aleph_0$-categorical structure $(\bZ,s)$, where $s$ is the successor map.
\end{exm}

\begin{exm}
  \label{exm:ZOrd}
  For a very similar non-example, consider the structure $(\bZ,\leq)$, in which the usual metric on $\bZ$ is definable as a $[0,\infty]$-valued predicate.
  This is a proper metric structure whose automorphism group is isomorphic to $\bZ$.
  However, $(\bZ,\leq)$ is not locally $\aleph_0$-categorical, since its theory admits a unique $1$-type but two non-isolated $2$-types, namely $x \ll y$ and $x \gg y$.
\end{exm}

A reduct of an $\aleph_0$-categorical theory is again $\aleph_0$-categorical.
The following show that a general reduct of a locally $\aleph_0$-categorical theory or structure need not be one.

\begin{exm}
  \label{exm:ZBiColour}
  Consider the action of $2\bZ \curvearrowright \bZ$ by translation, and let $M = \bZ^{2\bZ}$ (in the sense of \autoref{dfn:GroupActionStructure}).
  By \autoref{thm:GroupActionStructure}, $M$ is locally $\aleph_0$-categorical, and its theory has quantifier elimination.
  From this it is straightforward to check that $M$ is bi-definable with $(\bZ,s,P)$, where $P$ is a unary predicate for the even numbers.
  We may also define in it an equivalence relation $Q(x,y)$, which holds when $x$ and $y$ have the same parity.

  Consider the structure $M' = (\bZ,s,Q)$.
  It is a reduct of $M$, and the localising metric of $M$ is definable in $M'$ (since it is definable using $s$ and equality).
  However, $M'$ has only one $1$-type but two distinct $2$-types at infinite distance: one where $Q(x, y)$ hods and one where it does not.
  Therefore, $M'$ is not locally $\aleph_0$-categorical and neither is its theory.
\end{exm}

We show that a reduct of a locally $\aleph_0$-categorical structure is also locally $\aleph_0$-categorical under additional hypotheses.
In fact, these hypotheses imply that the structure and its reduct are bi-definable, so the applicability is more restricted.
Nonetheless, this allows us to produce important examples, first as structures of the form $M^G$, and then, using the following theorems, as structures in a more natural language.

\begin{thm}
  \label{thm:LocallyA0CategoricalReduct}
  Let $M$ be a locally $\aleph_0$-categorical structure, equipped with a localising metric, and let $M'$ be a reduct of $M$.
  Assume that $M$ and $M'$ share the same metric and automorphism group.
  Then $M'$ is locally $\aleph_0$-categorical, and is bi-definable with $M$.
\end{thm}
\begin{proof}
  Denote $T = \Th(M)$ and $T' = \Th(M')$.
  Let $G = \Aut(M) = \Aut(M')$.
  Let also $M^G$ be the structure constructed in \autoref{dfn:GroupActionStructure}.
  It is locally $\aleph_0$-categorical by \autoref{thm:GroupActionStructure}.
  By \autoref{cor:LocallyA0CategoricalDefinitionEquivalent}, $M^G$ and $M$ are bi-definable.

  For $R > 0$, define
  \begin{gather*}
    B_R = \bigl\{ p \in \tS_n(T) : d(x_i,x_j)^p < R \ \text{for all} \ i<j<n \bigr\}.
  \end{gather*}
  Since $d$ is localising, $B_R \subseteq \fI_n(T) = M^n \dslash G$ and it is bounded there, so $\overline{B}_R$ is metrically compact.
  The reduct map $\pi_n\colon \tS_n(T) \rightarrow \tS_n(T')$ is $1$-Lipschitz in the respective type metrics, so $B_R'' \coloneq \pi_n(\overline{B}_R)$ is metrically compact.
  Consequently, the metric topology and logic topology coincide on $B_R''$.
  On the other hand, $B_R' \coloneq \pi_n(B_R) \subseteq \tS_n(T')$ is open, since $d$ is definable in $T'$.
  Therefore, for every $p' \in B'_R$, the metric topology and logic topology of $\tS_n(T')$ coincide at $p'$, i.e., $p'$ is an isolated type of $T'$.

  Since $\bigcup_R B_R'$ consists of all $n$-types realised in $M'$, the latter is atomic, and is the prime model of $T'$.
  It is a reduct of $M^G$, and by \autoref{prp:GroupActionStructureMinimal}, the two are bi-definable.
  Therefore, $M'$ and $M$ are bi-definable, so $M'$ is locally $\aleph_0$-categorical as well.
\end{proof}

\begin{cor}
  \label{cor:LocallyA0CategoricalSufficient}
  Let $M$ be a metric $\cL$-structure, and let $G = \Aut(M)$.
  Assume moreover that $M \dslash G$ is compact, $M^n \dslash G$ is a proper metric space for every $n$, and that every atomic $\cL$-formula is properly $G$-invariant in $M$.
  Then $M$ is locally $\aleph_0$-categorical and its metric is localising.

  Special cases where the hypotheses hold include:
  \begin{enumerate}
  \item
    \label{item:LocallyA0CategoricalSufficientRoelckePrecompact}
    When $M \dslash G$ is compact, $G$ is locally Roelcke precompact, $G \curvearrowright M$ is coarsely proper, and every atomic $\cL$-formula is properly $G$-invariant in $M$.
  \item
    \label{item:LocallyA0CategoricalSufficientProper}
    When $M$ is proper as a metric space, $M \dslash G$ is compact, and every atomic $\cL$-formula is properly $G$-invariant in $M$.
  \item
    \label{item:LocallyA0CategoricalSufficientMetric}
    When $(M,d)$ is a pure metric space (i.e., $\cL = \{d\}$), $M \dslash G$ is compact, and $M^n \dslash G$ is proper for each $n$.
  \end{enumerate}
\end{cor}
\begin{proof}
  By \autoref{thm:GroupActionStructure}, $M^G$ is locally $\aleph_0$-categorical.
  By \autoref{thm:LocallyA0CategoricalDefinable}, every atomic $\cL$-formula is definable in $M^G$, so $M$ is a reduct of $M^G$.
  By \autoref{thm:LocallyA0CategoricalReduct}, $M$ and $M'$ are bi-definable, and $M$ is locally $\aleph_0$-categorical.

  For the special case \autoref{item:LocallyA0CategoricalSufficientRoelckePrecompact}, just apply \autoref{thm:LocallyRoelckePrecompactProper}.
  For \autoref{item:LocallyA0CategoricalSufficientProper}, it is enough to observe that if $M$ is proper, then so is $M^n \dslash G$.
  For \autoref{item:LocallyA0CategoricalSufficientMetric}, note that $d$ is always properly $G$-invariant.
\end{proof}

\begin{cor}
  \label{cor:pure-metric-spaces-loca0}
  Let $(A, d)$ be a complete, separable, unbounded metric space, and let $G = \Iso(A)$.
  Then $(A, d)$ is locally $\aleph_0$-categorical if and only if $A \dslash G$ is compact and $A^n \dslash G$ is proper for all $n \in \bN$.
\end{cor}
\begin{proof}
  One direction holds by \autoref{prp:LocallyA0CategoricalLocalisingUnbounded} and \autoref{cor:LocallyA0CategoricalAutGroup}, the other by \autoref{cor:LocallyA0CategoricalSufficient}\autoref{item:LocallyA0CategoricalSufficientMetric}.
\end{proof}

Let $\cL$ be a countable signature in $[-\infty,\infty]$-valued continuous logic.
In particular, the distance symbol $d$ may be $[0,\infty]$-valued, and every symbol has a prescribed uniform continuity modulus relative to $d$.
Let $\cL^\at(n)$ denote the collection of atomic formulas in $n$ variables, and let $\tS^\qf_n(\cL) \subseteq [-\infty,\infty]^{\cL^\at(n)}$ denote the collection of quantifier-free $n$-types in $\cL$, which is a compact space.
For any $\cL$-structure $M$, the quantifier-free type map $\tp^\qf\colon M^n \rightarrow \tS^\qf_n(\cL)$ is uniformly continuous relative to the uniform structure of $\tS^\qf_n(\cL)$ as a compact space.
It factors via a uniformly continuous map $\theta_n\colon M^n \dslash G \rightarrow \tS^\qf_n(\cL)$, where $G = \Aut(M)$.

An $\cL$-structure $M$ is \emph{quantifier-free homogeneous} if $\theta_n$ is injective for every $n$.
Equivalently, if for every $a,b \in M^n$ with the same quantifier-free type, there exist automorphisms that send $a$ arbitrarily close to $b$.
If there always exists an automorphism that sends $a$ to $b$, then $M$ is \emph{exactly quantifier-free homogeneous}.

\begin{cor}
  \label{cor:HomogeneousLocallyA0Categorical}
  Let $M$ be a quantifier-free homogeneous $\cL$-structure with (finite-valued) metric $d$ and let $G = \Aut(M)$.
  Suppose, moreover, that $M \dslash G$ is compact and that every atomic $\cL$-formula is properly $G$-invariant on $M$.
  Finally, suppose that for every $n$ and every closed bounded set $\Xi \sub M^n \dslash G$, the map $\theta_n^{-1} \colon \theta_n(\Xi) \rightarrow M^n \dslash G$ is uniformly continuous, where $\theta_n(\Xi)$ is equipped with the uniform structure induced from the compact space $\tS^\qf_n(\cL)$.

  Then $M$ is locally $\aleph_0$-categorical and the metric $d$ is localising.
\end{cor}
\begin{proof}
  Let $\Xi \sub M^n \dslash G$ be closed and bounded, and let $(p_k)_k$ be a sequence in $\Xi$.
  Then $ \bigl(\theta_n(p_k) \bigr)_k$ is a sequence in the compact space $\tS^\qf_n(\cL)$ so it admits a Cauchy subsequence.
  Therefore, $(p_k)_k$ admits a Cauchy subsequence.
  Since $\Xi$ is also complete, it is compact.

  Therefore $M^n \dslash G$ is proper, so \autoref{cor:LocallyA0CategoricalSufficient} applies.
\end{proof}

This observation is very similar to \cite[Theorem~7]{Zielinski:LocallyRoelckePrecompact}: the condition of \emph{pair-propinquity} defined there is a slightly restricted version of the last condition of \autoref{cor:HomogeneousLocallyA0Categorical}.

\begin{exm}
  The Urysohn space $\bU$ is quantifier-free homogeneous (and even exactly so).
  The last condition of \autoref{cor:HomogeneousLocallyA0Categorical} holds by Uspenskij \cite[Proposition~7.1]{Uspenskij:SubgroupsOfMinimalTopologicalGroups}.
  Therefore $\bU$ is locally $\aleph_0$-categorical and $\Iso(\bU)$ is locally Roelcke precompact.
  The fact that $\Iso(\bU)$ is locally Roelcke precompact was first proved by Rosendal~\cite{Rosendal:CoarseGeometry}.
\end{exm}

\begin{exm}
  The \emph{Urysohn diversity} was introduced by Bryant, Nies, and Tupper in \cite{Bryant-Nies-Tupper:UniversalDiversity}.
  Hallbäck~\cite{Hallback:PhD} proved that its automorphism group is locally Roelcke precompact, and the predicate symbols are properly $G$-invariant essentially for the same reason that the metric is.
  It follows from \autoref{cor:LocallyA0CategoricalSufficient}\autoref{item:LocallyA0CategoricalSufficientRoelckePrecompact} that it is locally $\aleph_0$-categorical.
\end{exm}

\begin{exm}
  \label{exm:proper-metric}
  If $(X,d)$ is a proper metric space with $X \dslash \Iso(X)$ compact, then condition \autoref{item:LocallyA0CategoricalSufficientProper} of \autoref{cor:LocallyA0CategoricalSufficient} holds.
  Therefore, $X$ is locally $\aleph_0$-categorical and its metric is localising.
  Furthermore, all equivalent conditions of \autoref{thm:LocallyCompact} hold in this case.
  Examples include Euclidean spaces, finite-dimensional hyperbolic spaces $\bH^n$, and quasi-transitive locally finite graphs (quasi-transitive means that the automorphism group of the graph has finitely many orbits).
\end{exm}

Many of the spaces in \autoref{exm:proper-metric} admit infinite-dimensional analogues.
For instance, all metrically homogeneous graphs are locally $\aleph_0$-categorical (see also \cite{Rosendal:CoarseGeometry}).

The infinite-dimensional hyperbolic space $\bH^\infty$ has been considered in many places in the literature and its isometry group was studied in detail by Duchesne~\cite{Duchesne:InfiniteDimensionalHyperbolicSpace}.
We will need the following three basic facts about it (see \cite{Burger-Iozzi-Monod:Hyperbolic} and \cite{Duchesne:InfiniteDimensionalHyperbolicSpace}).
\begin{itemize}
\item It is exactly quantifier-free homogeneous: every isometry between finite subsets of $\bH^\infty$ extends to an element of $G = \Iso(\bH^\infty)$.
  In particular, the action $G \curvearrowright \bH^\infty$ is transitive.
\item It is a geodesic metric space.
\item For every point $x \in \bH^\infty$, the stabiliser $G_x$ is isomorphic to the orthogonal group $O(\infty)$.
\end{itemize}

\begin{exm}
  \label{exm:HyperbolicSpace}
  The metric space $\bH^\infty$ is locally $\aleph_0$-categorical, its metric is localising, and its isometry group $G = \Iso(\bH^\infty)$ is locally Roelcke precompact.

  Indeed, let $x \in \bH^\infty$.
  Define $\varphi \colon \bH^\infty \dslash G_x \to [0, \infty)$ by $\varphi([y]) = d(x, y)$.
  It follows from the triangle inequality that $\varphi$ is a contraction.
  It is also surjective since $\bH^\infty$ is geodesic and unbounded.
  It is injective by exact homogeneity.
  Finally, assume that $0 \leq \alpha < \beta = d(x,y)$.
  Since $\bH^\infty$ is geodesic, there exists $y' \in \bH^\infty$ with $d(x, y') = \alpha$ and $d(y', y) = \beta - \alpha$.
  Then $\varphi(y) = \beta$, $\varphi(y') = \alpha$ and $d\bigl( [y], [y'] \bigr) = \beta - \alpha$.
  Therefore, $\varphi$ is an isometry.

  It follows that $\bH^\infty \dslash G_x$ is a proper metric space.
  Furthermore, $G_x$ is isomorphic to the orthogonal group, which is Roelcke precompact.
  By \autoref{prp:lRpcStabiliser}, $G$ is locally Roelcke precompact and $G \curvearrowright \bH^\infty$ is coarsely proper.
  By \autoref{cor:LocallyA0CategoricalSufficient}, $\bH^\infty$ is locally $\aleph_0$-categorical with a localising metric.

  The fact that $\Iso(\bH^\infty)$ is locally Roelcke precompact was first proved by Barritault.
\end{exm}

Banach spaces provide another important class of examples.
Let $\Iso(E)$ denote group of isometries of a Banach space $E$, and let $\Iso_0(E) \leq \Iso(E)$ denote the stabiliser of the origin.
By the Mazur--Ulam theorem \cite{Mazur-Ulam:TransformationsIsometriques}, $\Iso(E)$ coincides with the group of affine isometries of $E$, and $\Iso_0(E)$ with the group of linear isometries.

In order to view $E$ as a metric structure in the sense of continuous logic, one may consider the unit ball $E_1$, equipped with its metric convex structure (see \cite[2.1 Examples (4)]{BenYaacov-Berenstein-Henson-Usvyatsov:NewtonMS}).
This unit ball structure is what is generally referred to as a \emph{Banach space} in the context of continuous logic.
Its automorphism group is $\Iso_0(E)$.
It follows from Mankiewicz's theorem \cite{Mankiewicz:ExtensionsOfIsometries} that this group also coincides with the isometry group $\Iso_0(E_1)$ (i.e., the group of all isometries of $E_1$ fixing $0$).

Alternatively, one may consider the entire space $E$ as an unbounded metric space, whose automorphism group is $\Iso(E) = (E,+) \rtimes \Iso_0(E)$.
We will refer to this structure as an \emph{affine Banach space}.

\begin{thm}
  \label{thm:BanachSpaceExample}
  Let $E$ be a separable Banach space
  Then the following are equivalent:
  \begin{enumerate}
  \item
    The affine Banach space $E$ (i.e., $E$ as a pure metric space) is locally $\aleph_0$-categorical.
  \item
    The group $\Iso(E)$ is locally Roelcke precompact and the action $\Iso(E) \curvearrowright E$ is coarsely proper.
  \item
    The group $\Iso_0(E)$ is Roelcke precompact and $E_1 \dslash \Iso_0(E)$ is compact.
  \item
    The closed unit ball structure associated to $E$ (namely, $E_1$ with its metric convex structure) is $\aleph_0$-categorical.
  \item
    The pointed metric space $(E_1,0)$ is $\aleph_0$-categorical.
  \end{enumerate}
\end{thm}
\begin{proof}
  \begin{cycprf}
  \item[\eqnext] This follows from \autoref{cor:pure-metric-spaces-loca0} (and \autoref{thm:LocallyRoelckePrecompactProper}).
  \item[\eqnext] By \autoref{prp:lRpcStabiliser}, observing that $E \dslash \Iso_0(E)$ is a proper metric space if and only if $E_1 \dslash \Iso_0(E)$ is compact.
  \item[\eqnnext] This follows from \cite{BenYaacov-Tsankov:WAP} and Mankiewicz's theorem mentioned above.
  \end{cycprf}
\end{proof}

Examples of $\aleph_0$-categorical Banach spaces to which this applies are $L^p$ spaces (for $1 \leq p < \infty$) and the Gurarij space.
It is interesting to note that for $p \neq 2$, the groups $\Iso_0(L^p(\bR))$ are all isomorphic, so (the unit balls of) the spaces $L^p(\bR)$ are bi-interpretable.
This is not the case for the affine spaces: $L^p(X)$ and $L^q(X)$ are not coarsely equivalent for $p \neq q$ (see the discussion before Example~3.18 in \cite{Rosendal:CoarseGeometry}).
Therefore the affine Banach space $L^p(X)$ and $L^q(X)$ are not bi-interpretable for $p \neq q$, by \autoref{c:interpr-implies-coarse-eq}.
For the Gurarij space, the question whether its affine isometry group is locally Roelcke precompact is listed in \cite{Hallback:PhD} as an open problem.

\bibliographystyle{begnac}
\bibliography{LocCat}

\end{document}